\RequirePackage[l2tabu, orthodox]{nag}

\documentclass[12pt]{amsart}
\usepackage{amsmath, amsthm, thmtools, amssymb, colonequals, mathrsfs, environ,
nameref, stmaryrd, tikz, enumitem, fullpage}
\usetikzlibrary{arrows}

\usepackage[
  colorlinks,
  linkcolor={red!50!black},
  citecolor={blue!50!black},
  urlcolor={blue!80!black}
	backref,
	pdfauthor={Vishal Arul}, 
	pdftitle={Division by 1-zeta on Superelliptic Curves and Jacobians},
]{hyperref}

\usepackage[alphabetic,backrefs,lite,nobysame]{amsrefs} 

\DeclareMathOperator*{\Res}{Res}
\DeclareMathOperator{\Princ}{Princ}
\DeclareMathOperator{\Div}{Div}
\DeclareMathOperator{\sign}{sign}
\DeclareMathOperator{\characteristic}{char}
\DeclareMathOperator{\Sp}{Sp}
\DeclareMathOperator{\divisor}{div}
\DeclareMathOperator{\Gal}{Gal}
\DeclareMathOperator{\Spec}{Spec}
\DeclareMathOperator{\adj}{adj}
\DeclareMathOperator{\pr}{pr}
\DeclareMathOperator{\Norm}{Norm}
\DeclareMathOperator{\AJ}{AJ}

\usepackage{microtype}

\declaretheorem[name=Theorem,
style=theorem,
numberwithin=section,
refname={Theorem,Theorems},
Refname={Theorem,Theorems}]{mytheorem}

\declaretheorem[name=Lemma,
style=theorem,
sharenumber=mytheorem,
refname={Lemma,Lemmas},
Refname={Lemma,Lemmas}]{mylemma}

\declaretheorem[name=Corollary,
style=theorem,
sharenumber=mytheorem,
refname={Corollary,Corollaries},
Refname={Corollary,Corollaries}]{mycorollary}

\declaretheorem[name=Proposition,
style=theorem,
sharenumber=mytheorem,
refname={Proposition,Propositions},
Refname={Proposition,Propositions}]{myproposition}

\declaretheorem[name=Definition,
style=definition,
sharenumber=mytheorem,
refname={Definition,Definitions},
Refname={Definition,Definitions}]{mydefinition}

\declaretheorem[name=Remark,
style=remark,
sharenumber=mytheorem,
refname={Remark,Remarks},
Refname={Remark,Remarks}]{myremark}

\makeatletter
\renewcommand{\itemautorefname}{\@gobble}
\makeatother

\makeatletter
\renewcommand*{\eqref}[1]{%
  \hyperref[{#1}]{\textup{\tagform@{\ref*{#1}}}}%
}
\makeatother

\numberwithin{equation}{section}

\begin{document}
\title{}
\author{Vishal Arul}
\title{Division by $1 - \zeta$ on superelliptic curves and jacobians}
\thanks{This research was supported in part by a grant from the Simons Foundation (\#402472 to Bjorn Poonen).}

\begin{abstract}
Yuri Zarhin gave formulas for ``dividing a point on a hyperelliptic curve by
2.'' Given a point $P$ on a hyperelliptic curve $\mathcal{C}$ of genus $g$,
Zarhin gives the Mumford representation of an effective degree $g$ divisor $D$
satisfying $2(D - g \infty) \sim P - \infty$.

The aim of this paper is to generalize Zarhin's result to superelliptic curves;
instead of dividing by 2, we divide by $1 - \zeta$. There is no Mumford
representation for divisors on superelliptic curves, so instead we give formulas
for functions which cut out a divisor $D$ satisfying $(1 - \zeta) D \sim P -
\infty$.

Additionally, we study the intersection of $(1 - \zeta)^{-1} \mathcal{C}$ and
the theta divisor $\Theta$ inside the jacobian $\mathcal{J}$. We show that the
intersection is contained in $\mathcal{J}[1 - \zeta]$ and compute the
intersection multiplicities.
\end{abstract}

\maketitle
\tableofcontents
\newpage
\section{Introduction}

Fix coprime integers $n, d \ge 2$ and an algebraically closed field $K$ with
$\characteristic(K) \nmid n$. Let $\mathcal{C}$ be the smooth projective model
of the curve given by the equation
\begin{equation}
\label{Equation:EquationOfC}
y^n = (x + \alpha_1) \cdots (x + \alpha_d)
\end{equation}
where $\alpha_{1}, \cdots, \alpha_{d}$ are distinct elements of $K$. Then
$\mathcal{C}$ has a unique point at infinity, denoted by $\infty$. The genus of
$\mathcal{C}$ is
\[
g = \frac{1}{2}(n - 1)(d - 1).
\]

Every point of the jacobian $\mathcal{J}$ of $\mathcal{C}$ can be represented as
$[D - g \infty]$ for some effective degree $g$ divisor $D$.  Then $\mathcal{C}$
naturally embeds into $\mathcal{J}$ via the Abel--Jacobi map $P \mapsto [P -
\infty]$; that is, the point $P$ of $\mathcal{C}$ goes to the divisor class $[P
- \infty]$.  Given divisors $X$ and $Y$ on $\mathcal{C}$, we write ``$X \sim
Y$'' to indicate that $X$ is linearly equivalent to $Y$. Moreover, the notation
``$X \ge Y$'' means that $X - Y$ is effective. Define the ``$\gcd$'' of a
collection of divisors $\{ X_{i} \}$ to be the maximal $X$ such that $X \le
X_{i}$ for all $i$. See \cite{Fulton_1989,poonen2006lectures} for more details
about curves, their jacobians, and divisor classes.

Given a rational function $f$ on $\mathcal{C}$, we write 
\[
\divisor(f) \colonequals \sum_{P} v_{P}(f) P
\]
to denote the principal divisor associated to $f$ and 
\[
\divisor_{0}(f) \colonequals \sum_{P \colon v_{P}(f) \ge 0} v_{P}(f) P
\]
to denote the effective portion of $\divisor(f)$.

We use $\zeta$ to denote both the primitive $n$th root of unity in $K$ and the
automorphism $\zeta : \mathcal{C} \to \mathcal{C}$ which acts on points of
$\mathcal{C}$ via
\[
\zeta : (x, y) \mapsto (x, \zeta y).
\]
Then $\zeta$ also induces an automorphism of $\mathcal{J}$, which we will also
denote by $\zeta$.  Then $1 - \zeta$ is an endomorphism of $\mathcal{J}$.  

Our goal is to provide formulas for ``division by $1 - \zeta$'' for points of
$\mathcal{C}$. For a fixed point $P$ on $\mathcal{C}$, we seek to find rational
functions on $\mathcal{C}$ which cut out an effective degree $g$ divisor $D$
satisfying the property
\[
(1 - \zeta) [D - g \infty] = [P - \infty],
\]
which is equivalent to
\[
(1 - \zeta) D \sim P - \infty.
\]

When $n = 2$, the curve $\mathcal{C}$ is hyperelliptic and we seek to divide by
$1 - \zeta = 2$. Let $\iota$ be the hyperelliptic involution on $\mathcal{C}$.
In \cite{zarhin2019division}, Zarhin provides formulas for division by $2$ in
the hyperelliptic setting. His formulas are written in terms of the Mumford
representation (see \cite{MumfordTataII}, page 3.17). More specifically, Zarhin
finds two rational functions $f_{1}, f_{2}$ on $\mathcal{C}$ for which there
exist effective degree $g$ divisors $D$ and $E$ such that
\begin{align*} 
\divisor(f_{1}) &= D + \iota(E) - 2 g \infty \\ 
\divisor(f_{2}) &= D + E + \iota(P) - (2 g + 1) \infty.
\end{align*} 
From this, we get $(1 - \iota) D \sim P - \infty$, or equivalently, $2 (D -
g \infty) \sim P - \infty$.

In the superelliptic setting, there is no direct analogue of the Mumford
representation. Instead, we find $n$ rational functions $f_{1}, \cdots, f_{n}$
such that for some degree $g$ effective divisors $D$ and $E$,
\begin{align*}
\divisor(f_{1}) &= D + \zeta^{-1} (E) - 2 g \infty \\
\divisor(f_{2}) &= D + \zeta^{-2} (E) + \zeta^{-1}(P) - (2 g + 1) \infty \\
\vdots  \\
\divisor(f_{n}) &= D + E + \zeta^{-1}(P) + \zeta^{-2}(P) + \cdots
+ \zeta^{-(n - 1)} (P) - (2 g + n - 1) \infty.
\end{align*}
The first two equations yield $\divisor(f_1 / \zeta^{*} f_2) = (1 -
\zeta) D - (P - \infty)$, so $(1 - \zeta) D \sim P - \infty$.
Moreover, we will show that
\begin{equation}
\label{Equation:DisTheGCD}
D = \gcd_{1 \le j \le n} \divisor_{0} f_{j}.
\end{equation}

When $n = 2$, our formulas reduce to Zarhin's. However, Zarhin's techniques do
not readily extend from $n = 2$ to general $n$; the main obstruction is the lack
of a Mumford representation when $n > 2$.

\begin{itemize}
\item 
When $n = 2$, it is the case that 
\begin{align*}
f_1 &= U(x) \\
f_2 &= y - V(x)
\end{align*}
for some $U(x), V(x) \in K[x]$ satisfying $U | (V^2 - \prod (x + \alpha_i))$.
(The pair $(U, V)$ is called the Mumford representation of $D$.) Assuming that
$f_1, f_2$ are in this special format greatly simplifies the rest of the
computation.  However, even when $n = 3$, one cannot assume that $f_1, f_2$ will
have this special form; one must work with the more general $f_i = U_{0, i}(x) +
U_{1, i}(x) y + U_{2, i}(x) y^2$.
\item 
There are other ways to represent divisor classes on superelliptic curves; see
\cite{Galbraith2002} for another possible representation and algorithms for
computations in that representation. However, we were not able to use their
representation for our formulas.
\end{itemize}

As an application, we can divide any point $(-\alpha_i, 0)$ by $1 - \zeta$.
Since $[(-\alpha_i, 0) - \infty]$ generate $\mathcal{J}[1 - \zeta]$, we obtain
generators for $\mathcal{J}[(1 - \zeta)^2]$. In particular, for the case $n = 3$
we know that $\mathcal{J}[(1 - \zeta_3)^2] = \mathcal{J}[3]$, so our formulas
give a representation for each 3-torsion divisor class on a trigonal
superelliptic curve. We also hope that our formula can be used to perform
explicit descent and compute the rational points on some superelliptic curves.

One curious aspect of this formula is that whenever $P \neq \infty$, no $D$
satisfying $(1 - \zeta) D \sim P - \infty$ lands on the theta divisor $\Theta$
of the jacobian. That is, $\mathcal{C} \cap (1 - \zeta) \Theta = \{ 0 \}$, which
implies that $(1 - \zeta)^{-1} \mathcal{C} \cap \Theta = \mathcal{J}[1 -
\zeta]$.  In \autoref{Section:Intersection}, we compute the intersection
multiplicity of $(1 - \zeta)^{-1} \mathcal{C}$ and $\Theta$ at each point of
$\mathcal{J}[1 - \zeta]$. 

\section{The formula for division by \texorpdfstring{$1 - \zeta$}{1 - zeta}}
\label{Section:DivisionFormula}

Let $T$ be an $n \times n$ matrix. Let $T_{i, j}$ denote the $(i, j)$-th entry
of $T$.  The indices $i, j$ will be taken modulo $n$ to make sense of
expressions of the form $T_{-1, 2 n}$ (this means $T_{n - 1, n}$). The notation
$T^{(i, j)}$ represents the submatrix of $T$ obtained by removing the $i$th row
and $j$th column of $T$. The notation ``$\adj T$'' stands for the adjugate
matrix of $T$; its $(i, j)$-th entry is defined to be $(\adj T)_{i, j}
\colonequals (-1)^{i + j} \det T^{(j, i)}$. It is a fact that $T (\adj T) =
(\adj T) T = (\det T) I_{n}$.

\subsection{Statement of main result}
\label{Subsection:StatementMainDivision}
Suppose that $P = (a, b)$. By translating $P$ and $\mathcal{C}$, we may assume
that the the $x$-coordinate of $P$ is zero; that is, $P = (0, b)$.  Choose
$r_{i}$ such that
\begin{align*}
r_{i}^{n} &= \alpha_{i} \\
\prod r_{i} &= b
\end{align*}
Let $s_{j}$ be the $j$th elementary symmetric polynomial evaluated on the
$r_{i}$, where the convention is that $s_{m} = 0$ for $m \not\in [0, d]$.  (So
$b = s_{d}$.) For each $\ell \in \mathbf{Z}$, define 
\[
A_{\ell}(x) = \sum_{k \ge 0} (-1)^{(n - 1) k} s_{\ell - n k} x^{k} \in K[x].
\]
Let $A, Z, M, N$ be the following $n \times n$ matrices with entries in $K[x,
y]$.
\begin{align*}
A &\colonequals \begin{bmatrix}
A_{d} & A_{d - 1} & \cdots & A_{d - n + 2} & A_{d - n + 1} \\
A_{d + 1} & A_{d} & \cdots & A_{d - n + 3} & A_{d - n + 2} \\
\vdots & \vdots &\ddots &\vdots &\vdots \\
A_{d + n - 2} & A_{d + n - 3} & \cdots & A_{d} & A_{d - 1} \\
A_{d + n - 1} & A_{d + n - 2} & \cdots & A_{d + 1} & A_{d}
\end{bmatrix} \\
Z &\colonequals \begin{bmatrix}
\zeta^{0} & 0 & \cdots & 0 & 0 \\
0 & \zeta^{-1} & \cdots & 0 & 0 \\
\vdots &\vdots &\ddots &\vdots &\vdots \\
0 & 0 & \cdots & \zeta^{-(n-2)} & 0 \\
0 & 0 & \cdots & 0 & \zeta^{-(n - 1)}
\end{bmatrix} \\
M &\colonequals A - y Z \\
N &\colonequals \adj M.
\end{align*}
The goal is to prove the following theorem.
\begin{mytheorem} \label{Theorem:Main}
The divisor 
\[
D \colonequals \gcd_{1 \le j \le n} \divisor_{0} N_{1, j}
\]
is an effective degree $g$ divisor on $\mathcal{C}$ such that
\[
(1 - \zeta) D \sim P - \infty.
\]
\end{mytheorem}
\begin{proof}
We will prove this theorem at the end of \autoref{Subsection:DivisionMainProof}.
\end{proof}

\subsection{Computational lemmas}
As mentioned before, view the entries of $A, Z, M, N$ as elements of $K[x, y]$.

\begin{mydefinition}
Define $\sigma$ to be the automorphism of $K[x, y]$ over $K[x]$ sending $y
\mapsto \zeta^{-1} y$. 
\end{mydefinition}

Now we seek to understand how $\sigma$ operates on the entries of $M$ and $N$.
We do so in \autoref{Lemma:N-Relations}, and the following notation makes it
easier to express those relations.

\begin{mydefinition}
Define 
\[
\delta_{i, j} \colonequals \begin{cases}
1 &\text{if } i \equiv j \pmod{n} \\
0 &\text{otherwise.} \\
\end{cases}
\]
\end{mydefinition}

\begin{mylemma} \label{Lemma:N-Relations}
We have 
\begin{align}
M_{i + 1, j + 1} &= ( (-1)^{n - 1} x)^{\delta_{j, n} - \delta_{i, n}} \cdot
\sigma M_{i, j} \label{Equation:M-PlusOne} \\
N_{i + 1, j + 1} &= ( (-1)^{n - 1} x)^{\delta_{j, n} - \delta_{i, n}} \cdot
\sigma N_{i, j} \label{Equation:N-PlusOne}.
\end{align}
Equivalently, if $C$ is the $n \times n$ matrix
\[
C = 
\left[
\begin{array}{c|c} 
0 & I_{n - 1} \\\hline 
(-1)^{n - 1} x & 0 
\end{array}
\right]
\]
(where the $I_{n - 1}$ block is the $(n - 1) \times (n - 1)$ identity matrix)
then
\begin{align}
\sigma M = C M C^{-1} \label{Equation:M-PlusOne-Conjugate} \\
\sigma N = C N C^{-1} \label{Equation:N-PlusOne-Conjugate}.
\end{align}
\end{mylemma}
\begin{proof}
\eqref{Equation:M-PlusOne} follows from the fact that for $\ell \ge d + 1$,
$A_{\ell} = (-1)^{n - 1} x A_{\ell - n}$ and the fact that for $i, j \in [1,
n]$, $M_{i, j} = A_{d + i - j} - \delta_{i, j} \zeta^{1 - i} y$.
\eqref{Equation:M-PlusOne} is equivalent to
\eqref{Equation:M-PlusOne-Conjugate}.  Both $\sigma$ and conjugation commute
with the $\adj$-operation, so taking $\adj$ of both sides of
\eqref{Equation:M-PlusOne-Conjugate} gives \eqref{Equation:N-PlusOne-Conjugate}.
\eqref{Equation:N-PlusOne-Conjugate} is equivalent to
\eqref{Equation:N-PlusOne}.
\end{proof}

\begin{mylemma} \label{Lemma:QijMakesSense} 
$N_{1, j}$ lies in the ideal 
\[
\left(x, \prod_{k = j}^{n - 1} (y - \zeta^k s_d)\right)
\]
of $K[x, y]$.
\end{mylemma}

\begin{proof}
Since $A_{\ell} \equiv 0 \pmod{x}$ whenever $\ell \not\in [0, d]$,
\[
N_{1, j} = (-1)^{j + 1} \det M^{(j, 1)} \equiv (-1)^{j + 1} \det
\left[\begin{array}{c|c} U & V \\\hline 0 & W \end{array}\right] \pmod{x},
\]
where
\begin{align*}
U &= \begin{bmatrix}
s_{d - 1} & s_{d - 2} & s_{d - 3} & \cdots & s_{d - j + 1} \\
s_{d} - \zeta^{-1} y & s_{d - 1} & s_{d - 2} & \cdots & s_{d - j + 2} \\
& s_{d} - \zeta^{-2} y & s_{d - 1} & \cdots & s_{d - j + 3} \\
 & & \ddots & \ddots & \vdots  \\
& & & s_{d} - \zeta^{2 - j} y & s_{d - 1}
\end{bmatrix}\\
V &= \begin{bmatrix}
s_{d - j} & s_{d - j - 1} & \cdots & s_{d - n + 1} \\
s_{d - j + 1} & s_{d - j} & \cdots & s_{d - n + 2} \\
s_{d - j + 2} & s_{d - j + 1} & \cdots & s_{d - n + 3} \\
\vdots & \vdots & \ddots & \vdots \\
s_{d - 2} & s_{d - 3} & \cdots & s_{d- n + j - 1}
\end{bmatrix}\\
W &= \begin{bmatrix}
s_{d} - \zeta^{- j} y & s_{d - 1} & \cdots & s_{d- n + j + 1} \\
&  s_{d} - \zeta^{- (j + 1)} y & \cdots & s_{d- n + j + 2} \\
&  & \ddots & \vdots \\
&  & & s_{d} - \zeta^{-(n - 1)} y \\
\end{bmatrix}.
\end{align*}
Hence $N_{1, j} \equiv (-1)^{j + 1} \det U \cdot \det W \pmod{x}$. Since $W$ is
upper triangular,
\[
\det W = \prod_{k = j}^{n - 1} (s_d - \zeta^{-k} y)
\]
which implies
\[
N_{1, j} \equiv (-1)^{j + 1} (\det U) \cdot \prod_{k = j}^{n - 1} (s_d -
\zeta^{-k} y) \pmod{x},
\]
as desired. 
\end{proof}

We work in a slightly larger ring $L$ where the eigenvalues of $A$ are defined.
\begin{mydefinition}
\label{Definition:Lsigma}
Define
\[
L \colonequals K[x, y, T] / (T^n + (-1)^n x) \simeq K[y, T]. 
\]
Then $\sigma$ extends to an automorphism of $L$ over $K[T]$ sending $y \mapsto
\zeta^{-1} y$.
\end{mydefinition}

\begin{mylemma} \label{Lemma:A-evs}
For $1 \le k \le n$, define
\[
\lambda_{k} \colonequals \prod_{i = 1}^{d} (r_{i} + \zeta^{k} T) \in K[T]
\subseteq L.
\]
Then the $\lambda_k$ are distinct and form the complete set of eigenvalues of
$A$.
\end{mylemma}
\begin{proof} The $\lambda_k$ are distinct because the $T^d$-coefficient of
$\lambda_k$ is $\zeta^{k d}$ and $d$ is coprime to $n$.

Now we show that each $\lambda_k$ is an eigenvalue of $A$ by showing that 
\[
v_k \colonequals \begin{bmatrix}
1 & \zeta^{k} T & \cdots & \zeta^{(n - 1) k} T^{n - 1} \end{bmatrix}^{\top}
\]
is
a corresponding eigenvector. We will show that $A v_k = \lambda_k v_k$ by
showing that their $j$th entries for $j \in [1, n]$ are the same. This will
complete the proof.

We first compute $(A v_k)_{j}$ as follows:
\begin{align*}
(A v_k)_{j} &= \sum_{i = 1}^{n} \zeta^{k (i - 1)} T^{i - 1} A_{d + j - i} \\
&= \sum_{i = 1}^{n} \zeta^{k (i - 1)} T^{i - 1} \sum_{m \ge 0} (-1)^{(n - 1) m}
s_{d + j - i - m n} x^{m} \\
&= \sum_{i = 1}^{n} \zeta^{k (i - 1)} T^{i - 1} \sum_{m \ge 0}  (-1)^{(n - 1) m}
s_{d + j - i - m n} (- (-T)^{n})^{m}  \\
&= \sum_{i = 1}^{n} \sum_{m \ge 0}  \zeta^{k (i - 1)} s_{d + j - i - m n}
 T^{i + n m - 1}  \\
&= \sum_{i = 1}^{n} \sum_{m \ge 0}  \zeta^{k (i + m n - 1)} s_{d + j - i - m n}
 T^{i + n m - 1}. 
\end{align*}
As $i$ and $m$ vary in the range $1 \le i \le n$ and $m \ge 0$, the quantity $i
+ m n$ represents every positive integer exactly once. However, $s_{d + j - i -
m n}$ will be zero whenever $i + m n - j \not\in [0, d]$. So we may perform the
change of coordinates $a \colonequals i + m n - j$ and turn this into the finite
sum
\begin{align*}
(A v_k)_{j} &= \sum_{a = 0}^{d} \zeta^{k(j + a - 1)} s_{d - a} T^{j + a - 1} \\
&= \zeta^{k (j - 1)} T^{j - 1} \sum_{a = 0}^{d} (\zeta^{k} T)^a s_{d - a} \\
&= \zeta^{k (j - 1)} T^{j - 1} \sum_{a = 0}^{d} (\zeta^{k} T)^{a}
\sum_{i_1 < i_2 < \cdots < i_{d - a}} r_{i_1} \cdots r_{i_{d - a}} \\
&= \zeta^{k (j - 1)} T^{j - 1} \prod_{i = 1}^{d} (r_i + \zeta^k T) \\
&= \zeta^{k (j - 1)} T^{j - 1} \lambda_k \\
&= (\lambda_k v_k)_{j}.
\end{align*}
Hence, $v_k$ is a nonzero eigenvector of $A$ with eigenvalue $\lambda_k$. Since
we have shown that $\{ \lambda_k \}$ are $n$ distinct eigenvalues of $A$, they
must be all the eigenvalues of $A$.
\end{proof}

\begin{mylemma} \label{Lemma:A,M-det}
We have
\begin{align*}
\det A &= \prod_{i = 1}^{d} (x + \alpha_{i}) \\
\det M &= \prod_{i = 1}^{d} (x + \alpha_{i}) - y^{n}.
\end{align*}
\end{mylemma}
\begin{proof}
The first equality comes directly from multiplying the eigenvalues computed in
\autoref{Lemma:A-evs} and by observing that
\[
\prod_{k = 0}^{n - 1} (r_{i} + \zeta^{k} T) = r_{i}^{n} - (-1)^{n} T^{n} =
\alpha_{i} + x.
\]
Observe that $\det M$ is a polynomial in $y$ of degree $n$ with leading term
$\prod_{i = 0}^{n - 1} (-\zeta^{i} y) = -y^{n}$.  By taking the determinant of
both sides of \eqref{Equation:M-PlusOne-Conjugate}, we deduce that $\det M$
is invariant under $\sigma$.  Therefore $\det M$ can have no other terms in $y$,
so it is of the form $\det M = q(x) - y^{n}$. By plugging in $y = 0$ we see that
$q(x) = \det (A - 0 \cdot Z) = \det A$, so the rest comes from the computation
of $\det A$.
\end{proof}

\begin{mylemma} \label{Lemma:N-rk-bound}
The determinant of any $2 \times 2$ submatrix of $N$ is divisible by $y^{n} - (x
+ \alpha_1) \cdots (x + \alpha_d)$.
\end{mylemma}
\begin{proof}
We show this for the submatrix of $N$ obtained by taking the $\{ i, k \}$ rows
and $\{ j, \ell \}$ columns. Let $F$ be the submatrix of $M$ obtained by
deleting the $\{ i, k \}$ rows and $\{ j, \ell \}$ columns.  Apply Jacobi's
complementary minor formula (Theorem 2.5.2 of \cite{prasolov1994problems}) with
these rows and columns to obtain
\[
\det \begin{bmatrix}
N_{i, j} & N_{i, \ell} \\
N_{k, j} & N_{k, \ell}
\end{bmatrix} = \pm \det M \cdot \det F.
\]
Since $-\det M = y^{n} - (x + \alpha_1) \cdots (x + \alpha_d)$ by
\autoref{Lemma:A,M-det}, we are done.
\end{proof}

For $t_x, t_y \in K$, define $A(t_x), M(t_x, t_y), N(t_x, t_y) \in
M_{n}(K)$ by substituting $x = t_x$ and $y = t_y$.
\begin{mylemma} \label{Lemma:A-rk-Bound}
For any $t_x \in K$, the rank of $A(t_x)$ is at least $n - 1$.
\end{mylemma}
\begin{proof}
The eigenvalues of $A$ were computed in \autoref{Lemma:A-evs}. Define
$T(t_x) \in K$ to be an $n$th root of $-(-1)^n t_x$ and define $\lambda_k(t_x)
\colonequals \prod_{i = 1}^d (r_i + \zeta^k T(t_x))$. Then the eigenvalues of
$A(t_x)$ are $\lambda_1(t_x), \cdots, \lambda_n(t_x)$.

\begin{enumerate}[label=\textbf{Case~\Alph*:}, ref={Case~\Alph*},
leftmargin=*, itemindent=25pt]
\item \label{Case:ArkBdtxnot0}
$t_x \neq 0$

Suppose that $\lambda_{k}(t_x) = \lambda_{\ell}(t_x) = 0$. Then there exist $i,
j$ such that $T(t_x) = - \zeta^{-k} r_{i}$ and $T(t_x) = - \zeta^{-\ell} r_{j}$.
Hence $\alpha_i = r_i^n = (- T(t_x))^{n} = r_j^n = \alpha_j$, so $i = j$.  Then
$\zeta^{k} = - r_{i} T(t_x)^{-1} = - r_{j} T(t_x)^{-1} = \zeta^{\ell}$, so $k =
\ell$. Hence $\lambda_{k}(t_x) = 0$ for at most one $k$, so the rank of $A(t_x)$
is at least $n - 1$.

\item \label{Case:ArkBdtx0}
$t_x = 0$

Since $A_{\ell} \equiv s_{\ell} \pmod{x}$ for all $\ell$ and $s_{\ell} = 0$ when
$\ell \not\in [0, d]$, we see that $A(0)$ is an upper triangular matrix with
diagonal entries $s_d$ and ``super-diagonal'' entries $s_{d - 1}$. If $s_d \neq
0$, then $A(0)$ is invertible and we are done.  If $s_{d} = 0$ and $s_{d - 1}
\neq 0$, then the submatrix obtained by deleting the first column and last row
of $A(0)$ is upper-triangular with diagonal entries $s_{d - 1}$ and is therefore
invertible, implying that the rank of $A(0)$ is at least $n - 1$.

If $s_d = s_{d - 1} = 0$, then at least two of the $\{ \alpha_1, \cdots,
\alpha_d \}$ are zero, which is impossible.  
\end{enumerate}
\end{proof}

\begin{mylemma} \label{Lemma:M-rk-Bound}
For any $t_x, t_y \in K$, the matrix $N(t_x, t_y)$ is not zero.
\end{mylemma}
\begin{proof}
We will use the following fact: for each square matrix $F$, the rank of $F$ is
at most $n - 2$ if and only if $\adj F = 0$.

Consider the matrix $N + \sigma N + \cdots + \sigma^{n - 1} N$; it is
$\sigma$-invariant and it involves powers of $y$ only between $0$ and $n - 1$,
so it is independent of $y$. Hence 
\begin{equation}
\label{Equation:TraceOfN}
(N + \sigma N + \cdots + \sigma^{n - 1} N)(x, y) = (N + \sigma N + \cdots +
\sigma^{n - 1} N)(x, 0) = n N(x, 0) = n \adj A(x).
\end{equation}

\begin{enumerate}[label=\textbf{Case~\Alph*:}, ref={Case~\Alph*},
leftmargin=*, itemindent=25pt]
\item \label{Case:MrkBdtxnot0}
$t_x \neq 0$

If $N(t_x, t_y) = 0$, then \eqref{Equation:N-PlusOne} implies that $(\sigma^i
N)(t_x, t_y) = 0$ for all $i$. Substituting this into \eqref{Equation:TraceOfN}
yields
\[
0 = (N + \sigma N + \cdots + \sigma^{n - 1} N)(t_x, t_y) = n \adj A(t_x).
\]
Since $\characteristic(K) \nmid n$, we may divide by $n$ on both sides to see
that $\adj A(t_{x}) = 0$, so $A(t_x)$ has rank at most $n - 2$, contradicting 
\autoref{Lemma:A-rk-Bound}.

\item \label{Case:MrkBdtx0}
$t_x = 0$

Then the matrix $M(0, t_y) = A(0) - t_y Z$ is upper triangular with diagonal
entries $s_d - t_y \zeta^{i}$. If $t_y \neq 0$, then these diagonal entries will
all be distinct; in particular, at most one is zero, so $M(0, t_y)$ will have
rank at least $n - 1$. If $t_y = 0$, then $M(0, t_y) = A(0)$ and we are done by
\autoref{Lemma:A-rk-Bound}. 
\end{enumerate}
\end{proof}

\subsection{Main proof}
\label{Subsection:DivisionMainProof}

\subsubsection{Vanishing loci of $N_{i, j}$}
We will now view entries of $N$ as elements of the function field
$K(\mathcal{C})$ when writing expressions of the form $\divisor N_{i, j}$ or
$\divisor_0 N_{i, j}$. In order to make sense of such expressions, we need to
check that $N_{i, j}$ reduces to a \textit{nonzero} element of $K(\mathcal{C})$.

\begin{mylemma} \label{Lemma:NegVInfinities} \hfill
\begin{enumerate}[label=\upshape(\arabic*),
ref=\autoref{Lemma:NegVInfinities}(\arabic*)]

\item \label{Lemma:NegVInfinitiesx} 
$-v_{\infty}(x) = n$

\item \label{Lemma:NegVInfinitiesy} 
$-v_{\infty}(y) = d$

\item \label{Lemma:NegVInfinitiesAell} 
For $\ell \ge 0$, 
\begin{equation}
\label{Equation:AellValInfinity}
-v_{\infty}(A_{\ell}) \le \ell,
\end{equation}
with equality holding if and only if $\ell \equiv 0 \pmod{n}$.

\item \label{Lemma:NegVInfinitiesMuv} 
For $1 \le u, v \le n$,
\begin{equation}
\label{Equation:MijValInfinity}
-v_{\infty}(M_{u, v}) \le d + u - v,
\end{equation}
with equality holding if and only if $u = v$ or $u - v \equiv -d \pmod{n}$.

\end{enumerate}
\end{mylemma}

\begin{proof}

\autoref{Lemma:NegVInfinitiesx} and \autoref{Lemma:NegVInfinitiesy} follow
directly from \eqref{Equation:EquationOfC}, the equation of $\mathcal{C}$.

\begin{enumerate}[label=\upshape(\arabic*),
ref={the proof of \autoref{Lemma:NegVInfinities}(\arabic*)}]
\setcounter{enumi}{2}

\item 
\label{LemmaProof:NegVInfinitiesAell}
Since 
\[
A_{\ell} = \sum_{k \ge 0} (-1)^{(n - 1) k} s_{\ell - n k} x^{k}
\]
and $s_{\ell - n k} = 0$ whenever $\ell - n k \not\in [0, d]$, 
\[
\deg_{x} A_{\ell} \le \lfloor \ell / n \rfloor,
\]
so by \autoref{Lemma:NegVInfinitiesx},
\[
-v_{\infty}(A_{\ell}) \le n \lfloor \ell / n \rfloor.
\]
Since $n \lfloor \ell / n \rfloor \le \ell$, we obtain
\eqref{Equation:AellValInfinity}. If $\ell \not \equiv 0 \pmod{n}$, then $n
\lfloor \ell / n \rfloor < \ell$, so the inequality must be strict. If $\ell
\equiv 0 \pmod{n}$, then the $x^{\ell / n}$-coefficient of $A_{\ell}$ is
$(-1)^{(n - 1) \ell / n} s_{0} = (-1)^{(n - 1) \ell / n} \neq 0$ and hence
$-v_{\infty}(A_{\ell}) = n(\ell / n) = \ell$.

\item
\label{LemmaProof:NegVInfinitiesMuv}
Since 
\[
M_{u, v} = A_{d + u - v} - \zeta^{1 - u} \delta_{u, v} y,
\]
\eqref{Equation:MijValInfinity} follows by breaking into cases depending on
whether or not $u = v$ and then applying \autoref{Lemma:NegVInfinitiesy} and
\autoref{Lemma:NegVInfinitiesAell}.  If $u \neq v$, then $M_{u, v} = A_{d + u -
v}$, so \autoref{Lemma:NegVInfinitiesAell} gives that equality holds in
\eqref{Equation:MijValInfinity} if and only if $u - v \equiv -d \pmod{n}$.  If
$u = v$, then equality holds in \eqref{Equation:MijValInfinity} because
$-v_{\infty}(A_{d}) < d$ (since $d \not\equiv 0 \pmod{n}$) and $-v_{\infty}(y) =
d$. 

\end{enumerate}
\end{proof}

\begin{mylemma} \label{Lemma:N-P} \hfill
\begin{enumerate}[label=\upshape(\arabic*), ref=\autoref{Lemma:N-P}(\arabic*)]

\item \label{Lemma:N-P-NegValInftyNij}
$-v_{\infty}( N_{i, j}) = 2 g + (i - 1) + (n - j)$. In particular, $N_{i,
j} \neq 0$.

\item \label{Lemma:N-P-div0Nij}
Each $N_{i, j}$ satisfies
\[
\divisor_{0} N_{i, j} \ge \sum_{k = j - n}^{i - 2} \zeta^{k} P
\]

\end{enumerate}
\end{mylemma}
\begin{proof}\hfill
\begin{enumerate}[label=\upshape(\arabic*), ref={the proof of
\autoref{Lemma:N-P}(\arabic*)}]

\item \label{LemmaProof:N-P-NegValInftyNij}
For every integer $k$, let $L(k \infty)$ be the subspace of
$K(\mathcal{C})$ consisting of meromorphic functions that are holomorphic
everywhere except at $\infty$ and whose valuation at $\infty$ is at least $-k$.
Define $\ell \colonequals 2 g + (i - 1) + (n - j)$. 

Label the rows of $M^{(j, i)}$ by $\{ 1, 2, \cdots, j - 1, j + 1,
\cdots, n \}$ and the columns by $\{ 1, 2, \cdots, i - 1, i + 1, \cdots, n \}$.
We remind the reader that row and column indices are taken modulo $n$.

Expand $\det M^{(j, i)}$ as a sum over permutations
\[
\det M^{(j, i)} = \sum_{\substack{\sigma \in S_{n} \\ \sigma(j) =
i}} \sign(\sigma) M_{1, \sigma(1)} \cdots M_{j - 1, \sigma(j - 1)} M_{j + 1,
\sigma(j + 1)} \cdots M_{n, \sigma(n)}.
\]
For every $\sigma \in S_n$ satisfying $\sigma(j) = i$, apply
\eqref{Equation:MijValInfinity} to the summand corresponding to $\sigma$ to get
\begin{align*}
&-v_{\infty}(\sign(\sigma) M_{1, \sigma(1)} \cdots M_{j - 1, \sigma(j - 1)} M_{j
+ 1, \sigma(j + 1)} \cdots M_{n, \sigma(n)}) \\
&= \sum_{\substack{1 \le k \le n \\ k \neq j}} -v_{\infty}(M_{k, \sigma(k)}) \\
&\le \sum_{\substack{1 \le k \le n \\ k \neq j}} (d + k - \sigma(k)) \\
&= -(d + j - i) + \sum_{k = 1}^{n} d + (k - \sigma(k)) \\
&= -(d + j - i) + n d \\
&= \ell
\end{align*}
and hence
\[
-v_{\infty}(\det M^{(j, i)}) \le \ell. 
\]
Furthermore, $\det M^{(j, i)} \pmod{L((\ell - 1) \infty)}$ will be unchanged if
we replace the $(u, v)$-entry of $M$ with zero whenever we do not have equality
in \eqref{Equation:MijValInfinity}. That is, the $n \times n$ matrix
$\widetilde{M}$ defined by
\[
\widetilde{M}_{u, v} = \begin{cases}
M_{u, v} &\text{if } u - v \in \{ 0, -d \} \pmod{n}, \\
0 &\text{otherwise}
\end{cases}
\]
satisfies
\begin{equation}
\label{Equation:PasstoMtilde}
\det M^{(j, i)} \equiv \det \widetilde{M}^{(j, i)} \pmod{L((\ell - 1)\infty)}.
\end{equation}

\textbf{Claim.} Let $u \in [0, n - 1]$ be the unique integer such that $j \equiv
i + u d \pmod{n}$. Then 
\begin{align}
\det \widetilde{M}^{(j, i)} = \pm &M_{i, i + d} \cdots M_{i + (u - 1) d, i + u
d} \nonumber \\
\times &M_{i + (u + 1) d, i + (u + 1) d} \cdots M_{i + (n - 1) d, i + (n - 1)
d} \label{Equation:tildeMdet}
\end{align}

\textbf{Proof of claim.} Write
\begin{equation}
\label{Equation:ExpandtildeMji}
\det \widetilde{M}^{(j, i)} = \sum_{\substack{\sigma \in S_{n} \\ \sigma(j) =
i}} \sign(\sigma) \widetilde{M}_{1, \sigma(1)} \cdots \widetilde{M}_{j - 1,
\sigma(j - 1)} \widetilde{M}_{j + 1, \sigma(j + 1)} \cdots \widetilde{M}_{n,
\sigma(n)}.
\end{equation}
Suppose that $\sigma \in S_{n}$ satisfies $\sigma(j) = i$ and $\sigma(m) \in \{
m, m + d \} \pmod{n}$ for every $m \in [1, n] \setminus \{ j \}$; otherwise, the
summand corresponding to $\sigma$ in \eqref{Equation:ExpandtildeMji} is zero.
Then:

\begin{enumerate}[label=\upshape(\roman*)]
\item $\sigma(i + k d) = i + (k + 1) d$ for $k \in [0, u - 1]$.

Induct on $k$. If $k = 0$, then $u \neq 0$ and hence $i \neq j =
\sigma^{-1}(i)$, so $\sigma(i) \neq i$. Since $\sigma(i) \in \{ i, i + d \}$,
this forces $\sigma(i) = i + d$. Now suppose that $\sigma(i + k d) = i + (k + 1)
d$ for some $k \in [0, u - 2]$. Then $i + (k + 1) d \neq i + k d = \sigma^{-1}(i
+ (k + 1) d)$, so $\sigma(i + (k + 1) d) \neq i + (k + 1) d$. Since $\sigma(i +
(k + 1) d) \in \{ i + (k + 1) d, i + (k + 2) d \}$, this forces $\sigma(i + (k +
1) d) = i + (k + 2) d$.

\item  $\sigma(i - k d) = i - k d$ for $k \in [1, n - u - 1]$.

Induct on $k$. If $k = 1$, then $u \neq n - 1$ and hence $i - d \neq j =
\sigma^{-1}(i)$, so $\sigma(i - d) \neq i$. Since $\sigma(i - d) \in \{i - d, i
\}$, this forces $\sigma(i - d) = i - d$. Now suppose that $\sigma(i - k d) = i
- k d$ for some $k \in [1, n - u - 2]$. Then $i - (k + 1) d \neq i - k d =
\sigma^{-1}(i - k d)$, so $\sigma(i - (k + 1) d) \neq i - k d$. Since $\sigma(i
- (k + 1) d) \in \{ i - (k + 1) d, i - k d \}$, this forces $\sigma(i - (k + 1)
d) = i - (k + 1) d.$
\end{enumerate}

Properties (i) and (ii) uniquely determine $\sigma$, so \textbf{the proof of the
claim is complete.}

We attain the upper bound in \eqref{Equation:MijValInfinity} for
every term on the right hand side of \eqref{Equation:tildeMdet}, so applying
$-v_{\infty}$ to both sides of \eqref{Equation:tildeMdet} yields
\begin{equation}
\label{Equation:Mtildeval}
-v_{\infty}(\det \widetilde{M}^{(j, i)}) = \ell.
\end{equation}
Combining \eqref{Equation:PasstoMtilde} and \eqref{Equation:Mtildeval}, we
conclude that $-v_{\infty}(\det M^{(j, i)}) = \ell$. Since
$N_{i, j} = (-1)^{i + j} \det M^{(j, i)}$, we are done.

\item \label{LemmaProof:N-P-div0Nij}
Use \eqref{Equation:N-PlusOne} to reduce to the case $i = 1$.
\autoref{Lemma:N-P-NegValInftyNij} implies that  $N_{1, j}$ is not identically
zero, so applying $\divisor_0$ to \autoref{Lemma:QijMakesSense} (which makes
sense since polynomials in $x, y$ can only have poles at $\infty$) yields
\begin{align*}
\divisor_0 N_{1, j} &\ge \gcd\left\{ \divisor_0 x, \divisor_0 \prod_{k = j -
n}^{-1} (y - \zeta^{k} s_{d}) \right\} \\
&\ge \sum_{k = j - n}^{-1} \zeta^k P. 
\end{align*}
\end{enumerate}
\end{proof}

\begin{mydefinition}
\label{Def:QDE}
Define 
\begin{align*}
Q_{i, j} &\colonequals \divisor_{0} N_{i, j} - \sum_{k = j - n}^{i - 2}
\zeta^{k} P \\
D_{i} &\colonequals \gcd_{1 \le k \le n} Q_{i, k} \\
E_{j} &\colonequals Q_{1, j} - \gcd_{1 \le k \le n} Q_{1, k}
\end{align*}
By \autoref{Lemma:N-P-div0Nij}, $Q_{i, j} \ge 0$, so $D_{i} \ge 0$. Also, $E_{j}
\ge 0$.
\end{mydefinition}
Our first task is to translate the lemmas in the previous section to results
about the effective divisors $Q_{i, j}, D_{i}, E_{j}$.
\begin{mylemma} \label{Lemma:DiEj-Construction}
The effective divisors $D_{i}$, $E_{j}$ satisfy 
\[
D_{i} + E_{j} = Q_{i, j}.
\]
\end{mylemma}
\begin{proof}
Apply \autoref{Lemma:N-rk-bound} to the $2 \times 2$ submatrix of $N$ obtained
by taking rows $\{ i, k \}$ and columns $\{ j, \ell \}$ to obtain the equality
$N_{i, j} N_{k, \ell} = N_{i, \ell} N_{k, j}$ \textit{as elements of}
$K(\mathcal{C})$. Since the entries of $N$ have poles only at $\infty$, we may
take $\divisor_0$ of both sides to obtain $\divisor_0 N_{i, j} + \divisor_0
N_{k, \ell} = \divisor_0 N_{i, \ell} + \divisor_0 N_{k, j}$. Therefore,
\begin{equation}
\label{Equation:Q-Cross-Relations}
Q_{i, j} + Q_{k, \ell} = Q_{i, \ell} + Q_{k, j},
\end{equation}
and hence
\begin{align*}
D_{i} + E_{j} &= \left(\gcd_{1 \le k \le n} Q_{i, k}\right) + Q_{1, j} - \gcd_{1
\le k \le n} Q_{1, k} \\ 
&= \left( \gcd_{1 \le k \le n} (Q_{i, 1} - Q_{1, 1} + Q_{1, k})    \right) +
Q_{1, j} - \gcd_{1 \le k \le n} Q_{1, k}
&&\text{(by \eqref{Equation:Q-Cross-Relations})} \\
&= (Q_{i, 1} - Q_{1, 1}) + \left( \gcd_{1 \le k \le n} Q_{1, k} \right) + Q_{1,
j} - \gcd_{1 \le k \le n} Q_{1, k} \\
&= Q_{i, 1} - Q_{1, 1} + Q_{1, j} \\
&= Q_{i, j} &&\text{(by \eqref{Equation:Q-Cross-Relations}).} 
\end{align*}
\end{proof}

\begin{mylemma} \label{Lemma:gcd-DiEj-Effective}
We have
\[
\gcd_{1 \le i \le n} D_{i} = \gcd_{1 \le j \le n} E_{j} = 0.
\]
\end{mylemma}
\begin{proof}
If there existed a point $R$ on $\mathcal{C}$ such that $Q_{i, j} \ge R$ for all
$i, j$, then all the $N_{i, j}$ would vanish on $R$, which contradicts 
\autoref{Lemma:M-rk-Bound}. Therefore $0 \ge \gcd_{i, j} Q_{i, j}$. Since each
$Q_{i, j}$ is effective, we get the reverse inequality $\gcd_{i, j} Q_{i, j} \ge
0$. Hence
\[
\gcd_{1 \le i, j \le n} Q_{i, j} = 0.
\]
Taking $\gcd_{1 \le i, j \le n}$ of both sides of
\autoref{Lemma:DiEj-Construction} yields
\[
\gcd_{1 \le i \le n} D_{i} + \gcd_{1 \le j \le n} E_{j} = \gcd_{1 \le i, j \le
n} Q_{i, j}.
\]
Therefore $\gcd_{i} D_{i}$ and $\gcd_{j} E_{j}$ are effective divisors whose sum
is $0$; hence both are $0$.
\end{proof}
\begin{mylemma} \label{Lemma:DjD1}
For $1 \le i, j \le n$, 
\begin{align}
D_{i} &= \zeta^{i - 1} D_{1} \label{Equation:Dj-D-relationship} \\
E_{j} &= \zeta^{j - 1} E_{1}. \label{Equation:Ej-E-relationship}
\end{align}
\end{mylemma}
\begin{proof}
Taking $\divisor_0$ of both sides of \eqref{Equation:N-PlusOne} yields 
\[
\divisor_0 N_{i + 1, j + 1} = (\delta_{j, n} - \delta_{i, n}) \divisor_0 x +
\zeta \divisor_0 N_{i, j}. 
\]
Breaking into cases depending whether $i = n$ and/or $j = n$, we obtain
\[
Q_{i + 1, j + 1} = \zeta Q_{i, j},
\]
so by \autoref{Lemma:DiEj-Construction},
\begin{align}
\label{Equation:N-relation-to-Q-relation}
D_{i + 1} + E_{j + 1} &= \zeta D_{i} + \zeta E_{j}.
\end{align}
Taking $\gcd_{j}$ of both sides and applying 
\autoref{Lemma:gcd-DiEj-Effective} yields $D_{i + 1} = \zeta D_{i}$.  Similarly,
$E_{j + 1} = \zeta E_{j}$. 
\end{proof}

\begin{mydefinition}
\label{Def:DefDE}
Define $D \colonequals D_{1}$ and $E \colonequals E_{1}$.
\end{mydefinition}

We summarize our work in the following proposition.
\begin{myproposition} \label{Proposition:Zeroes-And-Poles}
For $1 \le i, j \le n$, 
\begin{align*}
\divisor_0 N_{i, j} &= \zeta^{i - 1} D + \zeta^{j - 1} E +
\left(\sum_{k = j - n}^{i - 2} \zeta^{k} P\right) \\
\divisor N_{i, j} &= \zeta^{i - 1} D + \zeta^{j - 1} E + \left(\sum_{k = j -
n}^{i - 2} \zeta^{k} P\right) - (2 g + (i - 1) + (n - j))\infty.
\end{align*}
\end{myproposition}
\begin{proof}
Combine \autoref{Def:QDE}, \autoref{Lemma:DiEj-Construction},
\eqref{Equation:Dj-D-relationship}, \eqref{Equation:Ej-E-relationship}, and
\autoref{Lemma:N-P-NegValInftyNij}.
\end{proof}

\subsubsection{Orders at infinity}
\begin{mylemma} \label{Lemma:Chicken-McNugget}
There is no $f \in K(\mathcal{C})^\times$ having a pole only at $\infty$ such
that the pole order at $\infty$ is $n d - n - d$.
\end{mylemma}
\begin{proof}
Let $R$ be the ring $R = K[x, y] / (y^{n} - \prod_{i = 1}^{d} (x +
\alpha_{i}))$; this is the affine coordinate ring of $\mathcal{C} \setminus \{
\infty \}$. A $K$-basis for $R$ is $\{ x^{a} y^{b} \colon 0 \le a \text{ and } 0
\le b \le n - 1 \}$; since \autoref{Lemma:NegVInfinitiesx} and
\autoref{Lemma:NegVInfinitiesy} imply $-v_{\infty}(x^{a} y^{b}) = n a + d b$
and $(d, n) = 1$ by assumption, each element of this basis has a different pole
order at $\infty$.  Therefore, the order of the pole at $\infty$ of any element
of $R$ is of the form $n a + d b$ for nonnegative $a, b$.

Suppose that $f \in K(\mathcal{C})^\times$ has a pole only at $\infty$. Then $f
\in R$.  From the previous paragraph, we have $-v_{\infty}(f) = n a + d b$ for
nonnegative $a, b$. If it were the case that $n a + d b = n d - n - d$, then $a
\equiv -1 \pmod{d}$ and $b \equiv -1 \pmod{n}$, so by nonnegativity of $a, b$ we
conclude that $a \ge d - 1$ and $b \ge n - 1$. But then
\[
n d - n - d = n a + d b \ge (n d - n) + (n d - d) = 2 n d - n - d,
\]
which is a contradiction.
\end{proof}

\begin{mydefinition}
\label{Definiton:AbelJacobiWr}
Define the Abel--Jacobi map
\begin{center}
\begin{tikzpicture}
\node(C) at (0, 1) {$\mathcal{C}$};
\node(J) at (3, 1) {$\mathcal{J}$};
\node(P) at (0, 0) {$P$};
\node(Pinf) at (3, 0) {$[P - \infty]$.};
\draw[->] (C) -- (J) node[midway, above] {$\AJ_{\infty}$};
\draw[|->] (P) -- (Pinf);
\end{tikzpicture}
\end{center}
For every $r \ge 1$, this induces a map $\mathcal{C}^r \to \mathcal{J}^r$.
Denote by $W_r$ the image of the composite morphism $\mathcal{C}^r \to
\mathcal{J}^r \to \mathcal{J}$, where the second map is the addition map. We
define $\Theta \colonequals W_{g - 1}$ to be the theta divisor.
\end{mydefinition}
\begin{mylemma} \label{Lemma:Thetag}
For $r \ge g$, $W_r = \mathcal{J}$.
\end{mylemma}
\begin{proof}
It is a simple consequence of the Riemann-Roch theorem that any degree zero
divisor on $\mathcal{C}$ has a representation as $[P_1 + ... + P_g - g \infty]$
for points $P_1, \dots, P_g$ of $\mathcal{C}$.
\end{proof}

The $n = 2$ case of the following theorem is Theorem 2.5 of
\cite{zarhin2019division} (on page 506).
\begin{mytheorem} \label{Theorem:Unique-Intersection-In-J}\hfill
\begin{enumerate}[label=\upshape(\arabic*),
ref=\autoref{Theorem:Unique-Intersection-In-J}(\arabic*)]

\item \label{Theorem:Unique-Intersection-In-J-1-z}
The intersection of $\AJ_{\infty}(\mathcal{C})$ and $(1 - \zeta) \Theta$ in
$\mathcal{J}$ is exactly $\{ 0 \}$.

\item \label{Theorem:Unique-Intersection-In-J-z-1}
The intersection of $\AJ_{\infty}(\mathcal{C})$ and $(\zeta - 1) \Theta$ in
$\mathcal{J}$ is also exactly $\{ 0 \}$.

\end{enumerate}
\end{mytheorem}
\begin{proof} \hfill
\begin{enumerate}[label=\upshape(\arabic*),
ref={the proof of \autoref{Theorem:Unique-Intersection-In-J}(\arabic*)}]

\item \label{TheoremProof:Unique-Intersection-In-J-1-z}
Suppose that there were some $P \in \mathcal{C} \setminus \{ \infty \}$ such
that $[P - \infty]$ lies in $(1 - \zeta) \Theta$. Then there is some effective
divisor $D$ of degree $r \le g - 1$ such that $(1 - \zeta) D \sim P - \infty$
and $v_{\infty}(D) = 0$. By \autoref{Lemma:Thetag}, there is an effective
divisor $E$ of degree $s \le g$ such that $D + E \sim (r + s) \infty$ and
$v_{\infty}(E) = 0$.  Define 
\[
t \colonequals (n d - n - d) - (r + s).
\]
Since $r \le g - 1$, $s \le g$, and $n d - n - d = 2 g - 1$, we have $t \ge 0$.
Consider the divisor
\[
F \colonequals \zeta^{t} D + E + \sum_{i = 0}^{t - 1} \zeta^{i} P.
\]
Since $E \sim (r + s) \infty - D$ and $P \sim \infty + (1 - \zeta) D$,
\begin{align*}
F &\sim \zeta^{t} D - D + \sum_{i = 0}^{t - 1} (\zeta^{i} D - \zeta^{i + 1} D) +
(r + s + t) \infty \\
&= 0 + (r + s + t) \infty \\
&= (n d - n - d) \infty.
\end{align*}
Since $v_{\infty}(F) = 0$ and $F \sim (n d - n - d) \infty$, this contradicts
\autoref{Lemma:Chicken-McNugget}.

\item \label{TheoremProof:Unique-Intersection-In-J-z-1}
Applying the previous part to $\zeta^{-1}$ instead of $\zeta$, we see that
$\mathcal{C} \cap (1 - \zeta^{-1}) \Theta = \{ 0 \}$. Applying $\zeta$ to both
sides gives $\zeta \mathcal{C} \cap (\zeta - 1) \Theta = \{ 0 \}$.  Since $\zeta
\mathcal{C} = \mathcal{C}$, we are done. 

\end{enumerate}
\end{proof}

\begin{mylemma} \label{Lemma:Zeroes-And-Poles}
\hfill
\begin{enumerate}[label=\upshape(\arabic*),
ref=\autoref{Lemma:Zeroes-And-Poles}(\arabic*)]

\item \label{Lemma:ZeroesAndPolesdegDE}
We have $\deg D = \deg E = g$.

\item \label{Lemma:ZeroesAndPolesNoWeierstrass}
The support of $D$ avoids $\{ (-\alpha_{1}, 0), \cdots, (-\alpha_{d}, 0), \infty
\}$. The same holds for $E$.

\end{enumerate}
\end{mylemma}
\begin{proof} \hfill
\begin{enumerate}[label=\upshape(\arabic*),
ref={the proof of \autoref{Lemma:Zeroes-And-Poles}(\arabic*)}]

\item \label{LemmaProof:ZeroesAndPolesdegDE}
Applying \autoref{Proposition:Zeroes-And-Poles} gives 
\begin{align}
\divisor(N_{1, n} / \zeta^{*} N_{1, n - 1}) &= \divisor N_{1, n} - \zeta
\divisor N_{1, n - 1} \nonumber\\
&= (D + \zeta^{-1} E - 2 g \infty) - \zeta (D + \zeta^{-2} E + \zeta^{-1} P - (2
g + 1) \infty) \nonumber\\
&= (1 - \zeta) D - (P - \infty). \label{Equation:1zDP}
\end{align}
Suppose that $\deg D \le g - 1$. Then $[(1 - \zeta) D] \in (1 - \zeta) \Theta$.
Since $[(1 - \zeta) D]$ also equals $[P - \infty] \in \AJ_{\infty}(\mathcal{C})
\setminus \{ 0 \}$, we have found an element of $(1 - \zeta) \Theta \cap \left(
\AJ_{\infty}(\mathcal{C}) \setminus \{ 0 \} \right)$, contradicting
\autoref{Theorem:Unique-Intersection-In-J-1-z}. Hence,
\begin{equation}
\label{Equation:degDatleastg}
\deg D \ge g.
\end{equation}
Similarly,
\begin{equation}
\label{Equation:z1EzP}
\divisor(\zeta^{2*} N_{1, n} / \zeta^{*} N_{2, n}) = (\zeta - 1) E - (\zeta P -
\infty), 
\end{equation}
and a similar argument with \autoref{Theorem:Unique-Intersection-In-J-z-1}
implies
\begin{equation}
\label{Equation:degEatleastg}
\deg E \ge g.
\end{equation}

Taking $i = 1$ and $j = n$ in \autoref{Proposition:Zeroes-And-Poles} yields $D +
\zeta^{-1} E - 2 g \infty = \divisor N_{1, n}$, so
\begin{equation}
\label{Equation:degDE}
\deg D + \deg E = 2 g.
\end{equation}
Combining \eqref{Equation:degDatleastg}, \eqref{Equation:degEatleastg}, and
\eqref{Equation:degDE} gives $\deg D = \deg E = g$, as desired.

\item \label{LemmaProof:ZeroesAndPolesNoWeierstrass}
Suppose that $R \in \{ (-\alpha_{1}, 0), \cdots, (-\alpha_{d}, 0), \infty \}$
and $D \ge R$. Then $(1 - \zeta)[D - R] \in (1 - \zeta) \Theta$. Since $R \in
\mathcal{J}[1 - \zeta]$, \eqref{Equation:1zDP} implies that $(1 - \zeta)[D - R]
= [P - \infty] \in \AJ_{\infty}(\mathcal{C}) \setminus \{ 0 \}$. Hence $(1
- \zeta)[D - R]$ is an element of $(1 - \zeta) \Theta \cap
\AJ_{\infty}(\mathcal{C}) \setminus \{ 0 \}$, contradicting
\autoref{Theorem:Unique-Intersection-In-J-1-z}.

Similarly, if $S \in \{ (-\alpha_{1}, 0), \cdots, (-\alpha_{d}, 0), \infty \}$
and $E \ge S$, then \eqref{Equation:z1EzP} implies $(\zeta - 1) [E - S] = [\zeta
P - \infty]$, so $(\zeta - 1) [E - S] \in (\zeta - 1) \Theta \cap
\AJ_{\infty}(\mathcal{C}) \setminus \{ 0 \}$, contradicting
\autoref{Theorem:Unique-Intersection-In-J-z-1}.  
\end{enumerate}

\end{proof}

\begin{mycorollary}
\label{Corollary:PnotleE}
$P \not\le E$.
\end{mycorollary}
\begin{proof}
Suppose that $P \le E$ and let $E' = \zeta^{- 1} E -
\zeta^{-1} P$, so that by \autoref{Lemma:ZeroesAndPolesdegDE}, $E'$ is an
effective degree $g - 1$ divisor on $\mathcal{C}$ satisfying
\begin{align*}
(\zeta - 1) E' &= (\zeta - 1) \zeta^{-1} E - (\zeta - 1) \zeta^{-1} P \\
&\sim (\zeta - 1)(2 g \infty - D) - (P - \zeta^{-1} P) &\text{(by
\autoref{Proposition:Zeroes-And-Poles} with } i = 1, j = n\text{)} \\
&= (1 - \zeta)(D) - (P - \zeta^{-1} P) \\
&\sim P - \infty - (P - \zeta^{-1} P) &\text{(by \eqref{Equation:1zDP})} \\
&= \zeta^{-1} P - \infty, 
\end{align*}
which contradicts \autoref{Theorem:Unique-Intersection-In-J-z-1}. 
\end{proof}

\begin{proof}[Proof of \autoref{Theorem:Main}]
We wish to check $\gcd_{1 \le j \le n} \divisor_{0} N_{1, j} = D$. Applying
\autoref{Proposition:Zeroes-And-Poles} with $i = 1$ and then taking $\gcd$
yields
\[
\gcd_{1 \le j \le n} \divisor_{0} N_{1, j} \ge D.
\]
For contradiction, suppose that $Q$ is a point on $\mathcal{C}$ such that
\[
Q \le \gcd_{1 \le j \le n} (\divisor_{0} N_{1, j} - D).
\]
Then \autoref{Proposition:Zeroes-And-Poles} implies that for all $j \in [1, n]$,
\begin{align}
\label{Equation:HypotheticalQ}
Q &\le \zeta^{j - 1} E + \sum_{k = j - n}^{- 1} \zeta^{k} P = E_{j} +
\sum_{k = j - n}^{- 1} \zeta^{k} P 
\end{align}
by \eqref{Equation:Ej-E-relationship}.  By \autoref{Lemma:gcd-DiEj-Effective},
there must be some $u \in [1, n]$ such that $Q \not\le E_{u}$. Then
\[
Q \le \sum_{k = u - n}^{-1} \zeta^{k} P,
\]
so $Q = \zeta^{v} P$ for some $v \in [u - n, -1]$. 

\begin{enumerate}[label=\textbf{Case~\Alph*:}, ref={Case~\Alph*},
leftmargin=*, itemindent=25pt]
\item \label{Case:PFixedByZeta}
$P$ is fixed by $\zeta$

Then $Q = \zeta^{v} P = P$.  Substituting $j = n$ into
\eqref{Equation:HypotheticalQ} produces $Q \le \zeta^{n - 1} E$, so we conclude
that $P = \zeta P = \zeta Q \le E$, contradicting \autoref{Corollary:PnotleE}.

\item \label{Case:PNotFixedByZeta}
$P$ is not fixed by $\zeta$

Then the $\zeta^{k} P$ are distinct. Applying \eqref{Equation:HypotheticalQ}
with $j = v + n + 1$ then gives
\[
Q \le \zeta^{v} E + \sum_{k = v + 1}^{- 1} \zeta^{k} P
\]
Since $Q = \zeta^{v} P$ and the $\zeta^{k} P$ are distinct, we conclude that
$\zeta^{v} P \le \zeta^{v} E$, which implies that $P \le E$, again contradicting
\autoref{Corollary:PnotleE}. 
\end{enumerate}
\end{proof}

\subsection{Varying the choice of \texorpdfstring{$r_{i}$}{r\_i}}

Recall that $(r_{1}, \cdots, r_{d})$ is any $d$-tuple of elements of $K$
satisfying
\begin{align*}
r_{i}^{n} &= \alpha_{i} \\
\prod r_{i} &= b.
\end{align*}

Write $\mathbf{r}$ to denote $(r_1, \dots, r_d)$.  Since the $D$ in 
\autoref{Theorem:Main} depends on the choice of $\mathbf{r}$, we will denote it
$D_{\mathbf{r}}$ from now on. For $\mathbf{a} = (a_1, \dots, a_d) \in
(\mathbf{Z} / n \mathbf{Z})^d$, write $\zeta^{\mathbf{a}} \mathbf{r}$ to denote
$(\zeta^{a_{1}} r_{1}, \ldots, \zeta^{a_{d}} r_{d})$.

Applying \autoref{Theorem:Main} with $b$ replaced by $\zeta^{-(a_1 + \dots +
a_d)} b$ and $\mathbf{r}$ replaced with $\zeta^{-\mathbf{a}} \mathbf{r}$, we
obtain
\[
(1 - \zeta) D_{\zeta^{-\mathbf{a}} \mathbf{r}} \sim \zeta^{-(a_{1} + \cdots +
a_{d})} P - \infty,
\]
so
\[
D_{\mathbf{r}} - \zeta^{a_{1} + \ldots + a_{d}} D_{\zeta^{-\mathbf{a}}
\mathbf{r}} \in \mathcal{J}[1 - \zeta].
\]
Our goal is to write down $D_{\mathbf{r}} - \zeta^{a_{1} + \ldots + a_{d}}
D_{\zeta^{-\mathbf{a}} \mathbf{r}}$ in terms of a basis for $\mathcal{J}[1 -
\zeta]$. First, we provide a description of $\mathcal{J}[1 - \zeta]$.

\begin{mydefinition}
For $1 \le i \le d$, define $P_{i} \colonequals (-\alpha_{i}, 0) \in
\mathcal{C}(K)$. Define $\mathcal{P} \colonequals \{ P_{1}, \dots, P_{d} \}$.
\end{mydefinition}

\begin{myproposition} 
\label{Proposition:SuperellipticProp611Poonen}
There is a split exact sequence of $\mathbf{Z} / n \mathbf{Z}$-modules
\begin{center}
\begin{tikzpicture}
\node(0m) at (0, 0) {$0$};
\node(1m) at (3, 0) {$\mathbf{Z} / n \mathbf{Z}$};
\node(2m) at (6, 0) {$(\mathbf{Z} / n \mathbf{Z})^{\mathcal{P}}$};
\node(3m) at (9, 0) {$\mathcal{J}[1 - \zeta]$};
\node(4m) at (12, 0) {$0$};
\draw[->] (0m) -- (1m);
\draw[->] (1m) -- (2m) node[midway, above] {$\Delta$};
\draw[->] (2m) -- (3m) node[midway, above] {$s$};
\draw[->] (3m) -- (4m);
\end{tikzpicture}
\end{center}
where
\begin{align*}
\Delta(1) &= (1,\; \dots,\; 1) \\
s(a_1, \dots, a_d) &= \sum_{i = 1}^d a_i [P_{i} - \infty].
\end{align*}
\end{myproposition}

\begin{proof} (c.f.~\cite{poonen2006lectures}, Proposition 6.1.1)

\begin{enumerate}[label=\textbf{Step \arabic*:}, ref={Step \arabic*},
leftmargin=*, itemindent=25pt]
\item \label{Step:S611PsIsWellDefined}
$s$\textit{ is well-defined.}

Each point in $\mathcal{P} \cup \{ \infty \}$ is fixed by $\zeta$, so
$[P_{i} - \infty] \in \mathcal{J}[1 - \zeta]$. The calculation
\begin{equation}
\label{Equation:WeierstrassPointsNTorsion}
\divisor(x + \alpha_i) = n P_{i} - n \infty
\end{equation}
shows that the divisor classes $[P_{i} - \infty]$ are $n$-torsion. 

\item \label{Step:S611PsDeltaZero}
$s \circ \Delta = 0$\textit{.}

This follows from $\divisor(y) = \displaystyle\sum_{i = 1}^{d} [P_{i} -
\infty]$.

\item \label{Step:S611Pkers}
$\ker(s)$\textit{ is generated by }$(1, \dots, 1)$\textit{.}

\item \label{Step:S611Pcokers}
$s$\textit{ is surjective.}

We modify the proof of Proposition {3.2} in \cite{schaefer1998computing} to
prove \autoref{Step:S611Pkers} and \autoref{Step:S611Pcokers} simultaneously.
Use $\Div^{0}$ to denote the degree-zero divisors on $\mathcal{C}$ and use
$\Princ$ to denote the subgroup of principal divisors.  The following are exact
sequences of $\mathbf{Z}[\zeta]$-modules.

\begin{center}
\begin{tikzpicture}
\node(mm) at (-3, 1) {$0$};
\node(0m) at (0, 1) {$\overline{K}^\times$};
\node(1m) at (3, 1) {$\overline{K}(\mathcal{C})^\times$};
\node(2m) at (6, 1) {$\Princ$};
\node(3m) at (9, 1) {$0$;};
\node(mp) at (-3, 0) {$0$};
\node(0p) at (0, 0) {$\Princ$};
\node(1p) at (3, 0) {$\Div^0$};
\node(2p) at (6, 0) {$\mathcal{J}$};
\node(3p) at (9, 0) {$0$.};
\draw[->] (mm) -- (0m);
\draw[->] (0m) -- (1m);
\draw[->] (1m) -- (2m);
\draw[->] (2m) -- (3m);
\draw[->] (mp) -- (0p);
\draw[->] (0p) -- (1p);
\draw[->] (1p) -- (2p);
\draw[->] (2p) -- (3p);
\end{tikzpicture}
\end{center}
We now apply group cohomology with the group $G = \langle \zeta \rangle$. 

\begin{enumerate}[label=\upshape{(\roman*)}, ref={\theenumi(\roman*)}]
\item \label{Step:S611PcokersGalCohoHilb90}
Since $G \simeq \Gal(\overline{K}(\mathcal{C}) / \overline{K}(x))$,
\begin{equation}
\label{Equation:H0KC}
H^0(\overline{K}(\mathcal{C})^\times) = \overline{K}(x)^\times
\end{equation}
and Hilbert's Theorem 90 yields
\begin{equation}
\label{Equation:H1KC}
H^1(\overline{K}(\mathcal{C})^\times) = 0.
\end{equation}

\item \label{Step:S611PcokersGalCohoTrivAction}
Since $\overline{K}^\times$ is a trivial $G$-module, 
\begin{align}
H^0(\overline{K}^\times) &= 0 \label{Equation:H0K} \\
H^1(\overline{K}^\times) &= \mu_{n}(\overline{K}) \label{Equation:HoddK} \\
H^2(\overline{K}^\times) &= \overline{K}^\times / \overline{K}^{\times n} = 0.
\label{Equation:H2K}
\end{align}
Substituting \eqref{Equation:H1KC} and \eqref{Equation:H2K} into
\begin{center}
\begin{tikzpicture}
\node(0m) at (0, 1) {$H^1(\overline{K}(\mathcal{C})^\times)$};
\node(1m) at (3, 1) {$H^1(\Princ)$};
\node(2m) at (6, 1) {$H^2(\overline{K}^\times)$};
\draw[->] (0m) -- (1m);
\draw[->] (1m) -- (2m);
\end{tikzpicture}
\end{center}
yields
\begin{equation}
\label{Equation:H1Princ}
H^1(\Princ) = 0.
\end{equation}

\item \label{Step:S611PcokersGalCohoTrivH0Princ}
Substituting \eqref{Equation:H0K}, \eqref{Equation:H0KC},
\eqref{Equation:HoddK}, \eqref{Equation:H1KC} into 
\begin{center}
\begin{tikzpicture}
\node(0m) at (0, 1) {$H^0(\overline{K}^\times)$};
\node(1m) at (3, 1) {$H^0(\overline{K}(C)^\times)$};
\node(2m) at (6, 1) {$H^0(\Princ)$};
\node(3m) at (9, 1) {$H^1(\overline{K}^\times)$};
\node(4m) at (12, 1) {$H^1(\overline{K}(\mathcal{C})^\times)$};
\draw[->] (0m) -- (1m);
\draw[->] (1m) -- (2m);
\draw[->] (2m) -- (3m);
\draw[->] (3m) -- (4m);
\end{tikzpicture}
\end{center}
yields
\begin{center}
\begin{tikzpicture}
\node(0m) at (0, 1) {$0$};
\node(1m) at (3, 1) {$\overline{K}(x)^\times$};
\node(2m) at (6, 1) {$H^0(\Princ)$};
\node(3m) at (9, 1) {$\mu_{n}(\overline{K})$};
\node(4m) at (12, 1) {$0$,};
\draw[->] (0m) -- (1m);
\draw[->] (1m) -- (2m);
\draw[->] (2m) -- (3m);
\draw[->] (3m) -- (4m);
\end{tikzpicture}
\end{center}
so since the image of $\divisor(y) \in H^{0}(\Princ)$ generates
$\mu_{n}(\overline{K})$, 
\begin{equation}
H^0(\Princ) \text{ is generated by } \{ \divisor(y) \} \cup \{ \divisor(u)
\colon u \in \overline{K}(x)^\times  \}.
\label{Equation:H0Princ}
\end{equation}

\item \label{Step:S611PcokersGalCohoTrivJ1mz}
We substitute \eqref{Equation:H1Princ} into the long exact sequence
\begin{center}
\begin{tikzpicture}
\node(mm) at (-3, 1) {$0$};
\node(0m) at (0, 1) {$H^0(\Princ)$};
\node(1m) at (3, 1) {$H^0(\Div^0)$};
\node(2m) at (6, 1) {$\mathcal{J}[1 - \zeta]$};
\node(3m) at (9, 1) {$H^1(\Princ)$};
\draw[->] (mm) -- (0m);
\draw[->] (0m) -- (1m);
\draw[->] (1m) -- (2m);
\draw[->] (2m) -- (3m);
\end{tikzpicture}
\end{center}
to obtain
\begin{center}
\begin{tikzpicture}
\node(mm) at (-3, 1) {$0$};
\node(0m) at (0, 1) {$H^0(\Princ)$};
\node(1m) at (3, 1) {$H^0(\Div^0)$};
\node(2m) at (6, 1) {$\mathcal{J}[1 - \zeta]$};
\node(3m) at (9, 1) {$0$.};
\draw[->] (mm) -- (0m);
\draw[->] (0m) -- (1m);
\draw[->] (1m) -- (2m);
\draw[->] (2m) -- (3m);
\end{tikzpicture}
\end{center}
The group $H^0(\Div^0)$ consists of the $\zeta$-fixed divisors, so it is
generated by $[P_{i} - \infty]$ and $\Norm(P - \infty)$ for arbitrary $P
\in \mathcal{C}(\overline{K})$. Observe that 
\begin{align*}
\divisor (x - x(P)) &= \Norm(P - \infty), \\
\divisor (y) &= \sum [P_{i} - \infty],
\end{align*}
so by \eqref{Equation:H0Princ}, $H^0(\Princ)$ is generated by $\sum
[P_{i} - \infty]$ and $\Norm(P - \infty)$ for arbitrary $P \in
\mathcal{C}(\overline{K})$. Therefore, the $[P_{i} - \infty]$ generate
$\mathcal{J}[1 - \zeta] \simeq H^0(\Div^0) / H^0(\Princ)$ and the only relation
is $\sum [P_{i} - \infty] = 0$.
\end{enumerate}

\item \label{Step:S611Psplits}
\textit{The exact sequence in the statement of
\autoref{Proposition:SuperellipticProp611Poonen} splits.}

The splitting is given by
\begin{center}
\begin{tikzpicture}
\node(0m) at (0, 1) {$(\mathbf{Z} / n \mathbf{Z})^{\mathcal{P}}$};
\node(1m) at (4, 1) {$\mathbf{Z} / n \mathbf{Z}$};
\node(0p) at (0, 0) {$(a_1, \ldots, a_d)$};
\node(1p) at (4, 0) {$d^{-1} \displaystyle\sum a_i$. };
\draw[->] (0m) -- (1m);
\draw[->] (0p) -- (1p);
\end{tikzpicture}
\end{center}
\end{enumerate}
\end{proof}

\begin{mycorollary}
\label{Corollary:RecallSuperellipticProp611Poonen}
The map 
\begin{center}
\begin{tikzpicture}
\node(0m) at (0, 1) {$(\mathbf{Z} / n \mathbf{Z})^d$};
\node(1m) at (3, 1) {$\mathcal{J}[1 - \zeta]$};
\node(0p) at (0, 0) {$\mathbf{a}$};
\node(1p) at (3, 0) {$\sum_{i = 1}^d a_i [P_i - \infty]$};
\draw[->] (0m) -- (1m);
\draw[|->] (0p) -- (1p);
\end{tikzpicture}
\end{center}
is surjective and its kernel is generated by $(1, \dots, 1)$.
\end{mycorollary}
\begin{proof}
This is an immediate consequence of
\autoref{Proposition:SuperellipticProp611Poonen}.
\end{proof}

The $n = 2$ case of the following theorem is Theorem 1.1 of
\cite{zarhin2020halves}.
\begin{mytheorem} \label{Theorem:AllDFormula}
For each $\mathbf{a} \in (\mathbf{Z} / n \mathbf{Z})^d$, 
\[
D_{\mathbf{r}} - \zeta^{a_{1} + \ldots + a_{d}} D_{\zeta^{-\mathbf{a}}
\mathbf{r}} \sim a_{1} P_1 + \cdots +
a_{d} P_d - \left(\sum a_{j}\right) \infty.
\]
\end{mytheorem}
\begin{proof}
By induction, it suffices to treat the case $\mathbf{a} = (1, 0, \ldots, 0)$. To
do so, we will first reformulate \autoref{Theorem:Main} in terms of a
family over an open subset of $\mathbf{A}_{K}^{d} = \Spec K[r_{1}, \ldots,
r_{d}]$.

Let $U$ be the open subset of $\mathbf{A}_{K}^{d} = \Spec K[r_{1}, \ldots,
r_{d}]$ given by removing every hyperplane of the form $r_i^n = r_j^n$. Let
$\mathscr{C}$ be the smooth proper family of superelliptic curves over $U$ given
by the equation
\[
y^{n} = \prod_{i = 1}^{d} \left( x + r_{i}^{n} \right).
\]
The morphism $\mathscr{C} \to U$ admits two sections of interest to us; the
``$\infty$ section'' sends $(r_1, \dots, r_d)$ to the point at $\infty$ on the
fiber and the ``$P$ section'' sends $(r_1, \dots, r_d)$ to the point $(0, r_1
\cdots r_d)$ on the fiber.  Let $\mathscr{J}$ be the relative jacobian of the
family $\mathscr{C}$ and embed $\mathscr{C}$ into $\mathscr{J}$ using the
Abel--Jacobi map induced by the $\infty$ section (denoted $\AJ_{\infty}$). We
seek to compare the two sections $D_{\mathbf{r}}$ and $\zeta
D_{\zeta^{-\mathbf{a}} \mathbf{r}}$ of the map $\mathscr{J} \to U$.  Here is a
diagram representing all the morphisms considered thus far.
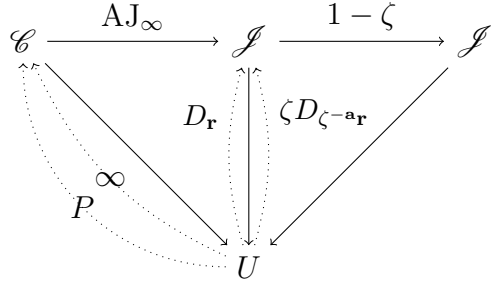
\begin{center}
\begin{tikzpicture}
\node(C) at (-3, 3) {$\mathscr{C}$};
\node(J1) at (0, 3) {$\mathscr{J}$};
\node(J2) at (3, 3) {$\mathscr{J}$};
\node(U) at (0, 0) {$U$};
\draw[->] (C) -- (U);
\draw[->] (J1) -- (U);
\draw[->] (J2) -- (U);
\draw[->] (C) -- node[above]{$\AJ_{\infty}$} (J1);
\draw[->] (J1) -- node[above]{$1 - \zeta$} (J2);
\draw[->, dotted] (U) to[bend left=20] node {$\infty$} (C);
\draw[->, dotted] (U) to[bend left=45] node {$P$} (C);
\draw[->, dotted] (U) to[bend left=15] node[pos=0.6, above left] {\small
$D_{\mathbf{r}}$} (J1);
\draw[->, dotted] (U) to[bend right=15] node[pos=0.6, above right] {\small
$\zeta D_{\zeta^{-\mathbf{a}} \mathbf{r}}$} (J1);
\end{tikzpicture}
\end{center}
In this language, \autoref{Theorem:Main} says that the following two
morphisms are the same:
\begin{center}
\begin{tikzpicture}
\node(U00) at (0, 0) {$U$};
\node(U01) at (0, 1) {$U$};
\node(J10) at (2.5, 0) {$\mathscr{J}$};
\node(C11) at (2.5, 1) {$\mathscr{C}$};
\node(J20) at (5, 0) {$\mathscr{J}$};
\node(J21) at (5, 1) {$\mathscr{J}$};
\draw[->] (U01) -- node[above] {$P$} (C11);
\draw[->] (C11) --node[above]{$\AJ_{\infty}$} (J21);
\draw[->] (U00) -- node[above] {$D_{\mathbf{r}}$} (J10);
\draw[->] (J10) -- node[above]{$1 - \zeta$} (J20);
\end{tikzpicture}
\end{center}
The map $\mathscr{J}[1 - \zeta] \to U$ is smooth of relative dimension 0; it is
\'{e}tale. Consider the sections $\gamma, \gamma' \colon U \to \mathscr{J}[1 -
\zeta]$ given in coordinates by 
\begin{align*}
\gamma &\colon (r_{1}, \ldots, r_{d}) \mapsto D_{\mathbf{r}} - \zeta
D_{\zeta^{-\mathbf{a}} \mathbf{r}} \\
\gamma' &\colon (r_{1}, \ldots, r_{d}) \mapsto (0, 0) - \infty.
\end{align*}
We wish to show that $\gamma = \gamma'$. Let $H_{1}$ be the hyperplane of $U$
cut out by $r_{1} = 0$. On $H_1$, we know that $\zeta^{-\mathbf{a}} \mathbf{r} =
\mathbf{r}$, so 
\[
D_{\mathbf{r}} - \zeta D_{\zeta^{-\mathbf{a}} \mathbf{r}} = (1 - \zeta)
D_{\mathbf{r}} \sim (0, 0) - \infty.
\]
Therefore, the sections $\gamma, \gamma'$ agree on the nonempty closed subset
$H_1$. Every section of an unramified cover with connected base is uniquely
determined by its image on a single point (by Corollaire 5.3, Expos\'e 1 of SGA
1 \cite{Grothendieck_1971}), so $\gamma = \gamma'$.
\end{proof}

\begin{myremark}
\autoref{Corollary:RecallSuperellipticProp611Poonen} and
\autoref{Theorem:AllDFormula} together imply that our formula in
\autoref{Theorem:Main} produces every effective degree $g$ divisor $D$
satisfying $(1 - \zeta) D \sim P - \infty$. 
\end{myremark}

\section{Application to the intersection of \texorpdfstring{$(1 - \zeta)^{-1}
\AJ_{\infty}(\mathcal{C})$}{(1 - zeta)\^{}\{-1\}AJ\_\{infty\}(C)} and
\texorpdfstring{$\Theta$}{Theta}}
\label{Section:Intersection}
Let $\mathcal{C}' \colonequals (1 - \zeta)^{-1} \AJ_{\infty}(\mathcal{C})$.
\autoref{Theorem:Unique-Intersection-In-J-1-z} implies that the intersection of
$\mathcal{C}'$ and $\Theta$ is contained in $\mathcal{J}[1 - \zeta]$. In this
section, we will compute the intersection multiplicities at each intersection
point. We will work over the complex numbers; that is, $K = \mathbf{C}$.  We
identify points of $\mathcal{J}$ with degree zero divisor classes, and in this
section, we use $D$ to denote degree zero divisor classes (as opposed to
effective divisors).

We define gaps as Nakayashiki does in \cite{Nakay2016}.
\begin{mydefinition}
For each point $P$ on $\mathcal{C}$ and degree zero divisor class $D$, define
\[
G_{P}(D) = \{ k \in \mathbf{Z}_{\ge 0} \colon h^{0}(D + k P) =
h^{0}( D +  (k - 1) P ) \}
\]
to be the set of gaps for $D$ at $P$.
\end{mydefinition}

\begin{mylemma}
\label{Lemma:WhereTheGapsAre}
$G_{P}(D)$ is a subset of $[0, 2 g - 1]$ of size exactly $g$.
\end{mylemma}
\begin{proof}
This is a straightforward consequence of the Riemann--Roch theorem.
\end{proof}

\begin{mydefinition}
\label{Definition:PartitionNakayashiki}
Suppose that $D \in \mathcal{J}$. By \autoref{Lemma:WhereTheGapsAre},
$G_{\infty}(D) = \{ b_{1}, \cdots, b_{g} \}$ for integers $0 \le b_{1} < b_{2} <
\ldots < b_{g} \le 2 g - 1$. As on page 5204 of \cite{Nakay2016}, define the
partition 
\[
\lambda_{D} \colonequals \left( b_{g} - (g - 1),\; b_{g - 1} - (g - 2),\;
\dots,\; b_{1} - 0 \right).
\]
Let $|\lambda_{D}|$ be the size of $\lambda_{D}$, i.e.,
\[
|\lambda_{D}| = \sum_{i = 1}^{g} \left( b_{i} - (i - 1) \right).
\]
\end{mydefinition}

\begin{mydefinition}
For each $D \in \mathcal{J}$, define $i(D)$ to be the intersection multiplicity
of $\mathcal{C}'$ and $\Theta$ at $D$.
\end{mydefinition}

The main theorem of this section is the following.
\begin{mytheorem}
\label{Theorem:Intersection-Multiplicities}
For each $D \in \mathcal{C}' \cap \Theta$,
\[
i(D) = |\lambda_{D}|.
\]
\end{mytheorem}
\begin{proof}
This theorem will be proved at the end of the section. 
\end{proof}

\begin{myremark}
We warn the reader that textbooks on Riemann surfaces \cite{Farkas_1992,
miranda1995algebraic} usually define gaps differently. For a point $P$ and
linear system $Q$ on $\mathcal{C}$, let $G'_{P}(Q)$ be the gaps for $Q$ at $P$
defined in Section VII.4 of \cite{miranda1995algebraic}. Let
$\omega_{\mathcal{C}}$ be the canonical bundle on $\mathcal{C}$, let
$\mathscr{O}_{\mathcal{C}}$ be the structure sheaf on $\mathcal{C}$, and let
$\mathcal{L}$ be the line bundle associated to $D$. Applying the Riemann--Roch
theorem shows that the relationship between the two notions of gaps is
\[
G_{P}(D) = \{ b \in \mathbf{Z}_{\ge 0} : b + 1 \in G'_{P}( \omega_{C} \otimes
\mathcal{L}^{-1} \otimes \mathscr{O}_{\mathcal{C}}(P)  ) \}
\]
and that $|\lambda_{D}|$ coincides with the inflectionary weight for $\omega_{C}
\otimes \mathcal{L}^{-1} \otimes \mathscr{O}_{\mathcal{C}}(\infty)$ at $\infty$. 
\end{myremark}

\begin{mydefinition}
Define the ring
\[
R \colonequals \mathbf{Z}[X_{1}^{\pm 1}, \dots, X_{d}^{\pm 1}] / (X_{1}^{n} -
1, \dots, X_{d}^{n} - 1, X_{1} \cdots X_{d} - 1).
\]
Then $R$ has a natural basis of the form $\{ X_{1}^{a_{1}} \cdots X_{d -
1}^{a_{d - 1}} \colon 0 \le a_{j} < n \}$. Define
\[
\pr_{a_{1}, \dots, a_{d - 1}} \colon R \to \mathbf{Z}
\]
to be the map that extracts the $X_{1}^{a_{1}} \cdots X_{d - 1}^{a_{d -
1}}$-coefficient. By abuse of notation, we also use $\pr_{a_{1}, \ldots, a_{d -
1}}$ to denote the same map, but tensored up to $\mathbf{Z}\llbracket T
\rrbracket$:
\[
\pr_{a_{1}, \ldots, a_{d - 1}} \colon R\llbracket T \rrbracket \to
\mathbf{Z}\llbracket T \rrbracket.
\]
Finally, define 
\begin{align*}
\rho &\colonequals (1 + T^{n} + T^{2n} + \cdots) \cdot \prod_{i = 1}^{d} (1 +
X_{i} T + \cdots + X_{i}^{n - 1} T^{n - 1}) &\in R\llbracket T \rrbracket \\
\rho_{a_{1}, \ldots, a_{d - 1}} &\colonequals \pr_{a_{1}, \ldots, a_{d -
1}}(\rho) &\in \mathbf{Z}\llbracket T \rrbracket.
\end{align*}
\end{mydefinition}

\begin{mylemma}
\label{Lemma:Torsion-Form}
Every element of $\mathcal{J}[1 - \zeta]$ has a unique representation of the
form
\[
[a_{1} P_{1} + \cdots + a_{d - 1} P_{d - 1} - \left( a_{1} + \cdots + a_{d -
1}\right) \infty]
\]
for some $0 \le a_{j} < n$. 
\end{mylemma}
\begin{proof}
This is an immediate consequence of
\autoref{Corollary:RecallSuperellipticProp611Poonen}.
\end{proof}

\begin{myproposition}
\label{Proposition:Power-Series-Multiplicities}
Suppose that $D = [a_{1} P_{1} + \cdots + a_{d - 1} P_{d - 1} - \left( a_{1} +
\cdots + a_{d - 1}\right) \infty]$ for some $0 \le a_{j} < n$. Then
\[
\rho_{a_{1}, \ldots, a_{d - 1}} = \sum_{i \in \mathbf{Z}_{\ge 0} \setminus
G_{\infty}(D)} T^{i}.
\]
\end{myproposition}
\begin{proof}
Writing out an explicit sum for $\rho$ gives
\begin{align*}
\rho &= \left( \sum_{m \ge 0} (T^{n})^{m} \right) \left( \sum_{e_{1}, \ldots,
e_{d}=0}^{n - 1} X_{1}^{e_{1}} \cdots X_{d}^{e_{d}} T^{e_{1} + \cdots + e_{n}}
\right) \\
&= \sum_{e_{1}, \ldots, e_{d} \in [0, n - 1], m \ge 0} X_{1}^{e_{1}} \cdots
X_{d}^{e_{d}} T^{e_{1} + \cdots + e_{d} + n m}
\end{align*}
Using the relation $X_{1} X_{2} \cdots X_{d} = 1$, the above equals
\[
\rho = \sum_{e_{1}, \ldots, e_{d} \in [0, n - 1], m \ge 0} X_{1}^{e_{1} - e_{d}}
\cdots X_{d - 1}^{e_{d - 1} - e_{d}} T^{(e_{1} - e_{d}) + \ldots + (e_{d - 1} -
e_{d}) + n m + d e_{d}},
\]
Perform the change of variables $a_{j} \equiv e_{j} - e_{d} \pmod{n}$ where
$a_{j} \in [0, n - 1]$. Then using $e_{j} = a_{j} + e_{d} - n \left\lfloor
\frac{a_{j} + e_{d}}{n}\right\rfloor$ and $X_{1}^{n} = \cdots = X_{d - 1}^{n} =
1$ yields
\[
\rho = \sum_{a_{1}, \ldots, a_{d - 1}, e_{d} \in [0, n - 1], m \ge 0}
X_{1}^{a_{1}} \cdots X_{d - 1}^{a_{d - 1}} T^{\left( \sum a_{j} \right) + n
\left( m - \sum \left\lfloor \frac{a_{j} + e_{d}}{n} \right\rfloor \right) + d
e_{d}}
\]
and hence
\begin{equation}
\label{Equation:Explicit-Form-Rho-Aj}
\rho_{a_{1}, \ldots, a_{d - 1}} = \sum_{e_{d} = 0}^{n - 1} \sum_{m \ge 0}
T^{\left( \sum a_{j} \right) + n
\left( m - \sum \left\lfloor \frac{a_{j} + e_{d}}{n} \right\rfloor \right) + d
e_{d}}
\end{equation}
For each $e_{d} \in [0, n - 1]$ and $m \ge 0$, define
\begin{equation}
\label{Equation:DefinitionOfEedm}
E(e_{d}, m) \colonequals \left( \sum a_{j} \right) + n \left( m - \sum \left\lfloor
\frac{a_{j} + e_{d}}{n} \right\rfloor \right) + d e_{d}
\end{equation}
to be the exponents arising in \eqref{Equation:Explicit-Form-Rho-Aj}. Observe
that $E(e_{d}, m)$ uniquely determines $e_{d}$ and $m$:
\begin{equation}
\label{Equation:Eedmmodn}
e_{d} \equiv d^{-1} \left( E(e_{d}, m) - \sum a_{j}\right) \pmod{n}
\end{equation}
uniquely determines $e_{d}$, and then
\[
m = \frac{1}{n} \left(E(e_{d}, m) - \left( \sum a_{j} \right) - d e_{d} \right)
+ \sum \left\lfloor \frac{a_{j} + e_{d}}{n} \right\rfloor
\]
is uniquely determined by $e_{d}$ and $E(e_{d}, m)$. Therefore, no terms in
\eqref{Equation:Explicit-Form-Rho-Aj} combine.

For each pair $(e_{d}, m)$, the function 
\[
h_{e_{d}, m} \colonequals y^{e_{d}} (x + \alpha_{d})^{m} \prod_{j = 1}^{d - 1}
(x + \alpha_{j})^{- \left \lfloor \frac{a_{j} + e_{d}}{n} \right \rfloor}
\]
satisfies
\begin{align*}
\divisor(h_{e_{d}, m}) &= (n m + e_{d}) P_{d} + \sum_{j = 1}^{d - 1} \left(
e_{d} - n \left\lfloor \frac{a_{j} + e_{d}}{n} \right\rfloor \right) P_{j} \\
&\quad- \left( n \left( m - \sum_{j = 1}^{d - 1} \left\lfloor \frac{a_{j} +
e_{d}}{n} \right\rfloor \right) + d e_{d} \right) \infty
\end{align*}
and hence
\begin{align*}
&\divisor(h_{e_{d}}, m) + \sum_{i = 1}^{d - 1} a_{i} (P_i - \infty)  \\
&= (n m +
e_{d}) P_{d} +  \sum_{j = 1}^{d - 1} \left( a_{j} + e_{d} - n \left\lfloor
\frac{a_{j} + e_{d}}{n} \right\rfloor \right) P_{j} \\
&\quad- \left( \left( \sum_{j = 1}^{d - 1} a_{j} \right) + n \left( m - \sum_{j =
1}^{d - 1} \left\lfloor \frac{a_{j} + e_{d}}{n} \right\rfloor \right) + d e_{d}
\right) \infty \\
&= (n m + e_{d}) P_{d} + \left( \sum_{j = 1}^{d - 1} \left( a_{j} + e_{d} - n
\left\lfloor \frac{a_{j} + e_{d}}{n} \right\rfloor \right) P_{j} \right) -
E(e_{d}, m) \infty,
\end{align*}
so  $E(e_{d}, m) \in \mathbf{Z}_{\ge 0} \setminus G_{\infty}(D)$.

To finish, we must check the reverse containment $\mathbf{Z}_{\ge 0} \setminus
G_{\infty}(D) \subseteq \{ E(e_{d}, m) : e_{d} \in [0, n - 1], m \ge 0\}$. By
\autoref{Lemma:NumberOfGapsIsG} below, $\# \{ E(e_{d}, m) : e_{d} \in [0, n -
1], m \ge 0\} = g$, so we are done.
\end{proof}
\begin{mylemma}
\label{Lemma:NumberOfGapsIsG}
Suppose that $a_{1}, \cdots, a_{d - 1} \in [0, n - 1]$. For each $e_{d} \in [0,
n - 1]$ and $m \ge 0$, define $E(e_{d}, m)$ as in
\eqref{Equation:DefinitionOfEedm}. Define
\[
S \colonequals \left\{ E(e_{d}, m) \colon e_{d} \in [0, n - 1], m \ge 0
\right\}.
\]
Then $\mathbf{Z}_{\ge 0} \setminus S$ is finite and has size exactly $g$.
\end{mylemma}
\begin{proof}

For any real number $x$, define $\{ x \} \colonequals x - \lfloor x \rfloor$.
Let $a_d = 0$ and $a = \sum_{j = 1}^{d} a_j$. For each $e \in [0, n - 1]$,
\eqref{Equation:Eedmmodn} implies that the subset of $S$ congruent to $a + d e
\pmod{n}$ is
\begin{align*}
S_e &\colonequals \{ E(e, m) : m \ge 0 \} \\
&= E(e, 0) + n \mathbf{Z}_{\ge 0} \\
&= \left( \left( \sum_{j = 1}^{d - 1} a_{j} \right) - n \left( 
\sum_{j = 1}^{d - 1} \left\lfloor \frac{a_{j} + e}{n} \right\rfloor
\right) + d e \right) + n \mathbf{Z}_{\ge 0} \\
&= \left(\sum_{j = 1}^{d} \left( a_{j} + e  - n \left\lfloor \frac{a_{j} + e}{n}
\right\rfloor \right)\right) + n \mathbf{Z}_{\ge 0} &\text{(since } a_{d} = 0
\text{)} \\
&= \left(n \sum_{j = 1}^{d} \left\{ \frac{a_{j} + e}{n} \right\}\right) + n
\mathbf{Z}_{\ge 0}.
\end{align*}
Therefore, 
\[
\# \left(\left( \mathbf{Z}_{\ge 0} \cap \left( a + d e + n \mathbf{Z} \right)
\right) \setminus S_{e} \right) = \left\lfloor \sum_{j = 1}^{d} \left\{
\frac{a_{j} + e}{n} \right\} \right\rfloor
\]
and hence
\begin{align}
\# \left(\mathbf{Z}_{\ge 0} \setminus S\right) &= \sum_{e = 0}^{n - 1}
\#\left(\left( \mathbf{Z}_{\ge 0} \cap \left( a + d e + n \mathbf{Z} \right)
\right) \setminus S_{e}\right)  \nonumber \\
&= \sum_{e = 0}^{n - 1} \left\lfloor \sum_{j = 1}^{d} \left\{ \frac{a_{j} +
e}{n} \right\}\right\rfloor \nonumber \\
&= \sum_{e = 0}^{n - 1} \left\lfloor \sum_{j = 1}^{d} \frac{a_{j} + e}{n} -
\left\lfloor \frac{a_{j} + e}{n} \right\rfloor\right\rfloor \nonumber \\
&= \sum_{e = 0}^{n - 1} \left( \left\lfloor \frac{a + d e}{n} \right\rfloor -
\sum_{j = 1}^{d} \left\lfloor \frac{a_{j} + e}{n} \right\rfloor \right)
\nonumber \\
&= \sum_{e = 0}^{n - 1} \left( \left( \frac{a + d e}{n} - \left\{ \frac{a + d
e}{n} \right\}\right)  - \sum_{j = 1}^{d} \left(\frac{a_{j} + e}{n} - \left\{
\frac{a_{j} + e}{n} \right\}\right) \right) \nonumber \\
&= \sum_{e = 0}^{n - 1} \left(  - \left\{ \frac{a + d e}{n} \right\}  + \sum_{j
= 1}^{d} \left\{ \frac{a_{j} + e}{n} \right\} \right) \nonumber \\
&= - \left(\sum_{e = 0}^{n - 1} \left\{ \frac{a + d e}{n} \right\}\right) +
\left(\sum_{j = 1}^{d} \sum_{e = 0}^{n - 1} \left\{ \frac{a_{j} + e}{n}
\right\}\right) \label{Equation:ExpandSizeOfComplementS}.
\end{align}
Note that the numbers $\{ a + d e \colon e \in [0, n - 1]\}$ hit each residue
class modulo $n$ exactly once. The same goes for $\{ a_{j} + e \colon e \in [0,
n - 1] \}$.  Hence,
\begin{equation}
\label{Equation:TakeAllFractionalPartsCircle}
\sum_{e = 0}^{n - 1} \left\{ \frac{a + d e}{n} \right\} = \sum_{e = 0}^{n - 1}
\left\{ \frac{a_{j} + e}{n} \right\} = \frac{0}{n} + \frac{1}{n} + \cdots +
\frac{n - 1}{n} = \frac{n - 1}{2},
\end{equation}
and substituting \eqref{Equation:TakeAllFractionalPartsCircle} into
\eqref{Equation:ExpandSizeOfComplementS} yields
\[
\# \left(\mathbf{Z}_{\ge 0} \setminus S\right) = -\left( \frac{n - 1}{2} \right)
+ \sum_{j = 1}^{d} \left( \frac{n - 1}{2} \right) = \frac{(n - 1)(d - 1)}{2} =
g. 
\]
\end{proof}

The next step is to extract $|\lambda_{D}|$ from $\rho_{a_{1}, \ldots, a_{d -
1}}$, which we will do in \autoref{Corollary:Power-Series-Weight}.

\begin{mydefinition}
For $h \in \mathbf{Z}\llbracket T \rrbracket$ and $i \ge 0$, write $[T^i] h$ to
denote the $T^i$-coefficient of $h$.
\end{mydefinition}

\begin{mycorollary}
\label{Corollary:Power-Series-Weight}
Keeping the notation of \autoref{Proposition:Power-Series-Multiplicities}, we
have
\[
|\lambda_{D}| + \frac{g(g - 1)}{2} = [T^{2 g}] \{ T^{2} (1 + T + \cdots)^{2}
\rho_{a_{1}, \ldots, a_{d - 1}} \}.
\]
\end{mycorollary}
\begin{proof}
We have
\begin{align*}
&[T^{2 g}] \{ T^{2} (1 + T + \cdots)^{2} \rho_{a_{1}, \ldots, a_{d - 1}} \}  \\
&= [T^{2 g - 1}] \{ T (1 + T + \cdots)^{2} \rho_{a_{1}, \ldots, a_{d - 1}} \} \\
&= [T^{2 g - 1}] \{ (T + 2 T^2 + 3 T^3 + \ldots) \rho_{a_{1}, \ldots, a_{d - 1}}
\} \\
&= \sum_{i \in [0, 2 g - 1] \setminus G_{\infty}(D)} (2 g - 1 - i)
&&\text{(by \autoref{Proposition:Power-Series-Multiplicities})} \\
&= g (2 g - 1) - \sum_{i \in [0, 2 g - 1] \setminus G_{\infty}(D)} i \\
&= g (2 g - 1) - \left( \left(\sum_{i = 0}^{2g -1} i\right) - \left( \sum_{i \in
G_{\infty}(D)} i \right) \right)  \\
&= \sum_{i \in  G_{\infty}(D)} i \\
&= |\lambda_{D}| + \frac{g (g - 1)}{2}. &&
\end{align*}
\end{proof}

\begin{mylemma}
\label{Lemma:W-Sum-Torsion}
We have
\[
\sum_{D \in \mathcal{J}[1 - \zeta]} |\lambda_{D}| = \frac{g(n + 1) n^{d -
1}}{12}.
\]
\end{mylemma}
\begin{proof}
Using \autoref{Lemma:Torsion-Form} to sum both sides of
\autoref{Corollary:Power-Series-Weight} over all $D \in\mathcal{J} [1 - \zeta]$
yields 
\[
\left(\sum_{D \in \mathcal{J}[1 - \zeta]} |\lambda_{D}| \right) + \frac{g(g - 1) n^{d - 1}}{2} = [T^{2 g}] \{ T^{2}(1 + T +
\cdots)^{2} \rho|_{X_{1} = \cdots = X_{n} = 1} \},
\]
and since 
\begin{align*}
\rho|_{X_{1} = \cdots = X_{n} = 1} &= (1 + T^{n} + T^{2 n} + \cdots) (1 + T +
\cdots + T^{n - 1})^{d} \\
&= (1 + T + T^{2} + \cdots)(1 + T + \cdots + T^{n - 1})^{d - 1},
\end{align*}
we have
\begin{equation}
\label{Equation:Total-Weight-Coefficient}
\left(\sum_{D \in \mathcal{J}[1 - \zeta]} |\lambda_{D}| \right) + \frac{g(g - 1) n^{d - 1}}{2} = [T^{2 g}] \{ T^{2}(1 + T +
\cdots)^{3} (1 + T + \cdots + T^{n - 1})^{d - 1} \}.
\end{equation}
Define $c_{i}$ so that
\begin{equation}
\label{Equation:Ci-Def}
\sum_{i = 0}^{(n - 1)(d - 1)} c_{i} T^{i} = (1 + T + \cdots + T^{n - 1})^{d -
1}.
\end{equation}
Since $2 g = (n - 1)(d - 1)$, 
\begin{equation}
\label{Equation:Coefficient-To-Sum}
[T^{2 g}] \{ T^{2} (1 + T + \cdots)^{3} (1 + T + \cdots + T^{n - 1})^{d - 1}
\} = \sum_{i = 0}^{2 g} \binom{2 g - i}{2} c_{i}.
\end{equation}
The 0th, 1st, and 2nd derivatives of \eqref{Equation:Ci-Def} are
\begin{align*}
&\sum_{i = 0}^{2 g} c_{i} T^{i} = (1 + T + \cdots + T^{n - 1})^{d - 1} \\
&\sum_{i = 0}^{2 g} i c_{i} T^{i - 1} = (d - 1)(1 + 2 T + 3 T^{2} + \ldots + (n
- 1)
T^{n - 2})(1 + T + \cdots + T^{n - 1})^{d - 2} \\
&\sum_{i = 0}^{2 g} i (i - 1) c_{i} T^{i - 2} \\
&= (d - 1)(2 + 6 T + \ldots + (n - 1)
(n - 2) T^{n - 3})(1 + T + \cdots + T^{n - 1})^{d - 2} \\
&+ (d - 1)(d - 2) (1 + 2 T + 3 T^{2} + \ldots + (n - 1)
T^{n - 2})^{2}(1 + T + \cdots + T^{n - 1})^{d - 3}
\end{align*}
Substituting $T = 1$ everywhere above gives
\begin{align*}
\sum_{i = 0}^{2 g} c_{i} &= n^{d - 1} \\
\sum_{i = 0}^{2 g} i c_{i} &= (d - 1) \left( \frac{n - 1}{2} \right) n^{d - 1}
\\ &= g n^{d - 1} \\
\sum_{i = 0}^{2 g} i (i - 1) c_{i} &= (d - 1) \left( \frac{(n-1)(n-2)}{3}
\right) n^{d - 1} + (d - 1)(d - 2) \left( \frac{n - 1}{2} \right)^{2} n^{d - 1}
\\
&= g \left( g + \frac{n - 5}{6} \right) n^{d - 1},
\end{align*}
so the right hand side of \eqref{Equation:Coefficient-To-Sum} is
\begin{align}
\sum_{i = 0}^{2 g} \binom{2 g - i}{2} c_{i} &= \frac{1}{2} \left(\sum_{i = 0}^{2
g} (i^{2} - i) c_{i}\right) - (2 g - 1) \left(\sum_{i = 0}^{2 g} i c_{i}\right)
+ g (2 g - 1) \left(\sum_{i = 0}^{2 g} c_{i}\right) \nonumber \\ 
&= \left( \frac{1}{2} g \left( g + \frac{n - 5}{6} \right) - (2 g - 1) g + g (2
g - 1)\right) n^{d - 1} \nonumber \\
&= \left( \frac{1}{2} g^{2} + \frac{g (n - 5)}{12} \right) n^{d - 1}
\label{Equation:Sum-Binomials}.
\end{align}
Combining \eqref{Equation:Total-Weight-Coefficient},
\eqref{Equation:Coefficient-To-Sum}, and \eqref{Equation:Sum-Binomials} finishes
the proof.
\end{proof}

\begin{mylemma}
\label{Lemma:No-Weight-Torsion}
If $D \not\in \Theta$, then $|\lambda_{D}| = 0$.
\end{mylemma}
\begin{proof}
Suppose that $D \not\in \Theta$. If $k \in [0, g - 1] \setminus G_{\infty}(D)$,
then there would be an effective degree $k$ divisor $E$ such that $D = [E - k
\infty] = [(E + g - 1 - k) \infty - (g - 1) \infty] \in \Theta$, which
contradicts the assumption that $D \not\in \Theta$.  Therefore $G_{\infty}(D) =
[0, g - 1]$, so $|\lambda_{D}| = 0$. 
\end{proof}

\begin{mylemma}
\label{Lemma:W-Sum}
We have
\[
\sum_{D \in \mathcal{C}' \cap \Theta} |\lambda_{D}|
\ge \frac{g(n + 1) n^{d - 1}}{12}.
\]
\end{mylemma}
\begin{proof}
We have
\begin{align*}
\sum_{D \in \mathcal{C}' \cap \Theta} |\lambda_{D}|
&\ge \sum_{D \in \mathcal{J}[1 - \zeta] \cap \Theta} |\lambda_{D}|
&&\text{(since } \mathcal{C}'\supseteq \mathcal{J}[1 - \zeta]\text{)} \\
&= \sum_{D \in \mathcal{J}[1 - \zeta]} |\lambda_{D}|
&&\text{(by \autoref{Lemma:No-Weight-Torsion})} \\
&= \frac{g (n + 1) n^{d - 1}}{12} &&\text{(by \autoref{Lemma:W-Sum-Torsion}).}
&&
\end{align*}
\end{proof}

For the next couple lemmas, we recall some notions of singular cohomology with
integral coefficients:
\begin{align*}
H^1(\mathcal{C}, \mathbf{Z}) &= H^1(\mathcal{J}, \mathbf{Z}) \simeq
\mathbf{Z}^{2 g} \\
H^*(\mathcal{J}, \mathbf{Z}) &= \bigwedge\nolimits^*(H^1(\mathcal{J},
\mathbf{Z})).
\end{align*}
The wedge product provides the cup product pairing $\smile$ on $H^*(\mathcal{J},
\mathbf{Z})$.  The automorphism $\zeta$ of $\mathcal{C}$ induces the pullback
automorphism $\zeta^{*}$ of $H^1(\mathcal{C}, \mathbf{Z})$.

\begin{mylemma}
\label{Lemma:CharPolyZetaH1}
The characteristic polynomial of $\zeta^{*}$ acting on $H^1(\mathcal{C},
\mathbf{Z})$ is
\[
\left( 1 + T + T^2 + \dots + T^{n - 1} \right)^{d - 1}.
\]
\end{mylemma}
\begin{proof}
We may as well verify this assertion after tensoring up to $\mathbf{C}$.
Applying Theorem 2.2 of \cite{dokchitser2018models} to $\mathcal{C}$ gives a
basis of $H^0(\mathcal{C}, \Omega^1)$, which we reindex and rescale to obtain
the following basis of $H^0(\mathcal{C}, \Omega^1)$:
\begin{equation}
\label{Eqn:H1CO1Basis}
\{ x^{a - 1} y^{-b} \, dx \colon  (a, b) \in \mathbf{Z}^{2} \colon 1 \le a \le d
- 1,\; 1 \le b \le n - 1,\; n a < d b \}.
\end{equation}
From the Hodge decomposition, $H^1(\mathcal{C}, \mathbf{Z}) \otimes \mathbf{C}
\simeq H^1_{dR}(\mathcal{C}, \mathbf{C}) = H^0(\mathcal{C},\Omega^1) \oplus
\overline{H^0(\mathcal{C},\Omega^1)}$, so a basis for $H^{1}(\mathcal{C},
\mathbf{C})$ consists of the $g$ holomorphic differentials in
\eqref{Eqn:H1CO1Basis} and their $g$ anti-holomorphic conjugates. Observe that
this basis is an eigenbasis for the $\zeta$-action, and that each eigenvalue is
of the form $\zeta^{k}$ for some $k \in [1, n - 1]$. 

For the eigenvalue $\zeta^{k}$, the corresponding eigenvectors are the
holomorphic differentials
\[ 
\{ x^{a - 1} y^{-(n - k)} \, dx \colon 1 \le a \le d - 1,\; n a < d (n - k) \} =
\left\{ x^{a - 1} y^{-(n - k)} \colon 1 \le a \le \frac{d (n - k)}{n} \right\} 
\]
and their antiholomorphic conjugates
\[
\{ \overline{x^{a - 1} y^{-k} \, dx} \colon 1 \le a \le d - 1,\; n a < d k \} =
\left\{ \overline{x^{a - 1} y^{-k}} \colon 1 \le a \le \frac{d k}{n} \right\} 
\]
Therefore, the total number of eigenvectors with eigenvalue $\zeta^k$ is
\[
\left \lfloor \frac{d (n - k)}{n} \right \rfloor +\left \lfloor \frac{d k}{n}
\right \rfloor = d - \left\lceil \frac{d k}{n} \right\rceil+\left \lfloor \frac{d
k}{n} \right\rfloor = d - 1.
\]
Hence each eigenvalue $\zeta^{k}$ has multiplicity $d - 1$, so the
characteristic polynomial is
\[
\prod_{k = 1}^{n - 1} (T - \zeta^k)^{d - 1} = \left( 1 + T + T^2 + \dots + T^{n
- 1} \right)^{d - 1}. 
\]
\end{proof}

\begin{mydefinition}
Denote the singular cohomology classes of the cycles $\{ 0 \},
\AJ_{\infty}(\mathcal{C})$, $\mathcal{C}'$, $\Theta$ on $\mathcal{J}$ by
$[\infty] \in H^{2 g}(\mathcal{J}, \mathbf{Z})$, $[\mathcal{C}], [\mathcal{C}']
\in H^{2g - 2}(\mathcal{J}, \mathbf{Z})$, $[\Theta] \in H^{2}(\mathcal{J},
\mathbf{Z})$ respectively.

For $r \ge 0$, define $W_{r}$ as in \autoref{Definiton:AbelJacobiWr} to
be the image of the composite morphism $\mathcal{C}^r \to \mathcal{J}^r
\stackrel{\text{sum}}{\to} \mathcal{J}$ and denote its cohomology class by
$[W_{r}] \in H^{2(g - r)}(\mathcal{J}, \mathbf{Z})$. In our notation, $[W_{0}] =
[\infty]$, $[W_{1}] = [\mathcal{C}]$, $[W_{g - 1}] = [\Theta]$ by
\autoref{Lemma:Thetag}.
\end{mydefinition}

\begin{mylemma}
\label{Lemma:SymplecticEigenbasis}
There is a $\mathbf{C}$-basis $\{ a_1, \dots, a_g, b_1, \dots, b_g \}$ for
$H^1(\mathcal{J}, \mathbf{C})$ such that
\begin{enumerate}[label=\upshape(\arabic*),
ref=\autoref{Lemma:SymplecticEigenbasis}(\arabic*)]

\item \label{Lemma:SymplecticEigenbasisSymplectic}
The basis is symplectic: for every $1 \le i, j \le n$, 
\begin{align*}
a_i \smile b_j &= \delta_{i, j}\\
a_i \smile a_j &= 0\\
b_i \smile b_j &= 0.
\end{align*}

\item \label{Lemma:SymplecticEigenbasisEigenbasis}
Each $a_i$ and $b_j$ is an eigenvector for $\zeta^{*}$.

\item \label{Lemma:SymplecticEigenbasisEigenvalues}
Let $\lambda(a_i)$, $\lambda(b_j)$ be the eigenvalues corresponding to $a_i$,
$b_j$, respectively. Then $\lambda(b_i) = \lambda(a_i)^{-1}$.
\end{enumerate}
\end{mylemma}

\begin{proof}
From symplectic linear algebra, each diagonalizable matrix $M$ in $\Sp(2 g,
\mathbf{C})$ has a symplectic eigenbasis. To see this, let $E_{\lambda}$ be the
eigenspace corresponding to the eigenvalue $\lambda \in \mathbf{C}$. Since $M$
respects the symplectic pairing, the eigenvalues come in pairs $\{ \lambda,
\lambda^{-1} \}$. For $\lambda \not\in \{ \pm 1 \}$, select any basis for
$E_{\lambda}$ and take the corresponding dual basis for $E_{\lambda^{-1}}$. For
$\lambda \in \{ \pm 1 \}$, the dimension of $E_{\lambda}$ must be even so one
may pick any symplectic basis for $E_{\lambda}$.

The lemma now follows from the observation in the previous paragraph since the
pullback of any automorphism of a manifold respects its cup product and 
\autoref{Lemma:CharPolyZetaH1} implies that the action of $\zeta^{*}$ on
$H^{1}(\mathcal{J}, \mathbf{C})$ is diagonalizable.
\end{proof}

\begin{mylemma}
\label{Lemma:CohoIntersection}
The following equality holds in $H^{0}(\mathcal{J}, \mathbf{Z})$:
\[
[\mathcal{C}'] \smile [\Theta] = \frac{g(n + 1) n^{d - 1}}{12} [\infty]. 
\]
\end{mylemma}
\begin{proof}
The following proof was suggested by Aaron Pixton.

We may as well verify this identity after tensoring up to $\mathbf{C}$. 
Let $\{ a_1, \dots, a_g, b_1, \dots, b_g \}$ be a $\mathbf{C}$-basis for
$H^1(\mathcal{J}, \mathbf{C})$ as in \autoref{Lemma:SymplecticEigenbasis}.

``Poincar\'e's Formula 11.2.1'' of \cite{Birkenhake_2004} implies that for
$r \in [0, g]$,
\[
[W_{r}] = \sum_{1 \le i_1 < i_2 < \dots < i_{g - r} \le g} (a_{i_1} \wedge
b_{i_1}) \wedge \cdots \wedge (a_{i_{g - r}} \wedge b_{i_{g - r}}),
\]
so
\begin{align*}
[\Theta] &= [W_{g - 1}] = \sum_{i = 1}^{g} a_{i} \wedge b_{i}  \\
[\mathcal{C}] &= [W_{1}] = \sum_{i = 1}^{g} a_{1} \wedge b_{1} \wedge
\cdots \wedge \widehat{a_{i} \wedge b_{i}} \wedge \cdots \wedge a_{g}
\wedge b_{g} \\
[\infty] &= [W_{0}] = a_{1} \wedge b_{1} \wedge \cdots \wedge a_{g} \wedge
b_{g}.
\end{align*}
(The hat indicates that the term is not there.)

Since $[\mathcal{C}'] = (1 - \zeta)^{*} [\mathcal{C}]$, a computation using
\autoref{Lemma:CharPolyZetaH1} and \autoref{Lemma:SymplecticEigenbasis} yields
\begin{align*}
[\mathcal{C}'] \smile [\Theta] &= \left((1 - \zeta)^{*} [\mathcal{C}]\right)
\smile [\Theta] \\
&= \left( \prod_{i = 1}^{g} (1 - \lambda(a_i)) (1 -\lambda(b_i)) \right) \cdot
\left(\sum_{i = 1}^{g} \frac{1}{(1 - \lambda(a_i))(1 - \lambda(b_i))}\right)
[\infty] \\
&= \left( \prod_{i = 1}^{n - 1} (1 - \zeta^i) \right)^{d - 1} \cdot \left(
\frac{g}{n - 1} \sum_{i = 1}^{n - 1} \frac{1}{(1 - \zeta^{i})(1 - \zeta^{-i})}
\right) [\infty] \\
&= n^{d - 1} \cdot \left(\frac{g}{n - 1} \cdot \frac{n^2 - 1}{12} \right)
[\infty] \quad \text{(from \autoref{Lemma:SumCsc2})}\\
&= \left(\frac{g (n + 1) n^{d - 1}}{12}\right) [\infty]. &&
\end{align*}
\end{proof}

\begin{mycorollary}
\label{Corollary:I-Sum}
We have
\[
\sum_{D \in \mathcal{C}' \cap \Theta} i(D) = \frac{g(n + 1) n^{d - 1}}{12}.
\]
\end{mycorollary}
\begin{proof}
The dual of \autoref{Lemma:CohoIntersection} implies that the total intersection
of $\mathcal{C}'$ and $\Theta$ in $\mathcal{J}$ is $g(n + 1) n^{d - 1} / 12$.
\end{proof}

\begin{mylemma}
\label{Lemma:SumCsc2}
We have
\[
\sum_{i = 1}^{n - 1} \frac{1}{(1 - \zeta^{i})(1 - \zeta^{-i})} = \frac{n^{2} -
1}{12}.
\]
\end{mylemma}
\begin{proof} 
The following proof was suggested by Bjorn Poonen. 

The differential $d(z^n-1)/(z^n-1)$ has a simple pole with residue $1$ at each
$n$th root of unity and a simple pole with residue $-n$ at infinity.  Therefore
the sum equals the sum of the residues of
\[
   \omega \colonequals \left(\frac{1}{(1-z)(1-z^{-1})}\right)
	 \frac{d(z^{n}-1)}{z^{n}-1}.
\]
at $n$th roots of unity not 1, or equivalently $- \Res_{\infty}(\omega) -
\Res_{1}(\omega)$.  Since $1/((1-z)(1-z^{-1}))$ vanishes at $\infty$,  $\omega$
is holomorphic at $\infty$.  On the other hand, Mathematica computes that
$\Res_{1}(\omega) = (1-n^2)/12$. 
\end{proof}

\begin{mydefinition}
Let 
\begin{align*}
&\omega_{\mathcal{C}} &&\text{be the canonical bundle of }\mathcal{C} \\
&V &&\text{be }H^{0}(\mathcal{C}, \omega_{\mathcal{C}}) \\
&\Lambda \subseteq V^{\vee} &&\text{be the period lattice of }\mathcal{C} \\
&z &&\text{be a local coordinate for }\mathcal{C}\text{ at } \infty \\
&A &&\text{be the set }\{ (a, b) \in \mathbf{Z}^{2} \colon 1 \le a \le d - 1,\; 1
\le b \le n - 1,\; n a < d b  \}. 
\end{align*}
We abuse notation and also use $z$ to denote a local coordinate for
$\AJ_{\infty}(\mathcal{C})$ at $0$.
\end{mydefinition}

\begin{mytheorem}
\label{Theorem:AbelJacobiWithIntegrals}
There is an isomorphism $\xi \colon \mathcal{J} \to V^{\vee} / \Lambda$
such that for all $P \in \mathcal{C}$, if $\gamma$ is a path on $\mathcal{C}$
from $\infty$ to $P$, then
\[
\xi \left( \AJ_{\infty}(P) \right) = \left( \kappa \in V \mapsto
\int_{\gamma} \kappa \in \mathbf{C} \right) \pmod{\Lambda} \in V^{\vee} /
\Lambda.
\]
\end{mytheorem}
\begin{proof}
See Section A.6.3 of \cite{Hindry_2000}.
\end{proof}

\begin{mydefinition}
Let $\pi$ be the composite $V^{\vee} \to V^{\vee} / \Lambda
\xrightarrow{\xi^{-1}} \mathcal{J}$. The kernel of $\pi$ is $\Lambda$ and $\pi$
expresses $V^{\vee}$ as the universal cover of $\mathcal{J}$.
\end{mydefinition}

\begin{mydefinition}
Define $\psi \colon A \to \mathbf{Z}$ by $\psi(a, b) = d b - n a$. For all $(a,
b) \in A$, note that $\psi(a, b) \ge 1$ and $\psi(a, b) = d b - na \le d (n - 1)
- n = 2g - 1$, so the image of $\psi$ is contained in $[1, 2 g - 1]$.
\end{mydefinition}

\begin{mylemma} \label{Lemma:BasisForV} \hfill
\begin{enumerate}[label={\upshape{(\arabic*)}},
ref={\autoref{Lemma:BasisForV}(\arabic*)}]

\item \label{Lemma:BasisForVInjection}
$\psi$ is an injection.

\item \label{Lemma:BasisForVBasis}
The set $\{ x^{a - 1} y^{-b} \, dx \colon (a, b) \in A \}$ is a basis of $V$. 

\item \label{Lemma:BasisForVExpansion}
There exist nonzero constants $C_{a, b}$ such that
\[
x^{a - 1} y^{-b} \, dx =  C_{a, b} z^{b d - a n - 1} \left( 1 + O(z) \right)\,
dz.
\]
\end{enumerate}
\end{mylemma}

\begin{proof} \hfill
\begin{enumerate}[label={\upshape{(\arabic*)}},
ref={the proof of \autoref{Lemma:BasisForV}(\arabic*)}]

\item \label{LemmaProof:BasisForVInjection}
If $\psi(a, b) = \psi(a', b')$, then $d (b - b') = n (a - a')$, and since $d$ is
coprime to $n$, this would imply $d | a - a'$, and since $a, a' \in [1, d - 1]$,
this means $a = a'$. Similarly, $b = b'$.

\item \label{LemmaProof:BasisForVBasis}
This is the basis we found in \eqref{Eqn:H1CO1Basis}.

\item \label{LemmaProof:BasisForVExpansion}
Since $-v_{\infty}(x) = n$ and $-v_{\infty}(y) = d$ by
\autoref{Lemma:NegVInfinitiesx} and \autoref{Lemma:NegVInfinitiesy}, there exist
constants $C_{x}$ and $C_{y}$ such that $x^{-1} = C_{x} z^{n} + O(z^{n + 1})$
and $y^{-1} = C_{y} z^{d} + O(z^{d + 1})$, so we are done by substituting these
into $x^{a - 1} y^{-b} \, dx$. 

\end{enumerate}
\end{proof}

In light of \autoref{Lemma:BasisForV}, we make the following definition.
\begin{mydefinition}
Let the image of $\psi$ be $\{ w_{1}, \cdots, w_{g} \}$ for $w_{1} < w_{2} <
\cdots < w_{g}$. Let $(a_{i}, b_{i})$ be the unique element of $A$ such that
$\psi(a_{i}, b_{i}) = w_{i}$. Define
\[
\kappa_{i} \colonequals C_{a_i, b_i}^{-1} x^{a_i - 1} y^{-b_{i}} \, dx.
\]
\end{mydefinition}

\begin{mycorollary}
\label{Corollary:Coordinatekappai}
The set $\{ \kappa_{1}, \cdots, \kappa_{g} \}$ is a basis for $V$ such that
\begin{equation}
\label{Equation:Coordinatekappai}
\kappa_i = z^{w_{i} - 1} \left( 1 + O(z) \right) \, dz.
\end{equation}
\end{mycorollary}
\begin{proof}
This is a restatement of \autoref{Lemma:BasisForVExpansion}.
\end{proof}

\begin{mydefinition}
Let $u_{w_{i}}$ be the coordinate function on $V^{\vee}$ associated to
$\kappa_{i}$, i.e., if $\langle \cdot, \cdot \rangle : V^{\vee} \times V \to
\mathbf{C}$ is the natural bilinear pairing, then for all $v \in V^{\vee}$,
$u_{w_{i}}(v) = \langle v, \kappa_{i} \rangle$. 
\end{mydefinition}

\begin{mydefinition}
Let $\zeta^{*}$ be the automorphism of $V$ induced by $\zeta$ and let
$\zeta_{*}$ be corresponding dual automorphism of $V^{\vee}$.
\end{mydefinition}

\begin{mylemma}
\label{Lemma:CoordinateFunctionPullbackZeta}
For all $v \in V^{\vee}$,
\[
u_{w_{i}}(\zeta_{*} v) = \zeta^{-b_{i}} u_{w_{i}}(v).
\]
\end{mylemma}
\begin{proof}
Since $\zeta$ acts on the function field of $\mathcal{C}$ by $\zeta^{*} x = x$
and $\zeta^{*} y = \zeta y$, it follows that $\zeta^{*} \kappa_{i} =
\zeta^{-b_{i}} \kappa_{i}$, and the lemma follows by taking the dual of this
relationship.
\end{proof}

\begin{mydefinition}
Let $U$ be a small simply-connected neighborhood of $0$ in
$\AJ_{\infty}(\mathcal{C})$. Note that $z \circ \AJ_{\infty}^{-1}$ is a local
coordinate on $U$; we will abuse notation and denote it by $z$. 

Since $1 - \zeta$ is a covering map, let $U'$ be the neighborhood of $0$ in
$\mathcal{C}'$ such that $(1 - \zeta)(U') = U$ and $(1 - \zeta)|_{U'} \colon U'
\to U$ is an isomorphism. Let $t = (1 - \zeta)^{*} z$ be a local coordinate on
$U'$.
\end{mydefinition}

We have the following commutative diagram.
\begin{center}
\begin{tikzpicture}
\node(p00) at (0, 0) {$\{ 0 \}$};
\node(p01) at (0, 2) {$\{ 0 \}$};
\node(U10) at (3, 0) {$U$};
\node(Up11) at (3, 2) {$U'$};
\node(C20) at (6, 0) {$\AJ_{\infty}(\mathcal{C})$};
\node(Cp21) at (6, 2) {$\mathcal{C}'$};
\node(J30) at (9, 0) {$\mathcal{J}$};
\node(J31) at (9, 2) {$\mathcal{J}$};
\draw[->] (p01) -- node[right] {$1 - \zeta$} (p00);
\draw[->] (Up11) -- node[right] {$1 - \zeta$} (U10);
\draw[->] (Cp21) -- node[right] {$1 - \zeta$} (C20);
\draw[->] (J31) -- node[right] {$1 - \zeta$} (J30);
\draw[right hook-latex] (p00) -- (U10);
\draw[right hook-latex] (U10) -- (C20);
\draw[right hook-latex] (C20) -- (J30);
\draw[right hook-latex] (p01) -- (Up11);
\draw[right hook-latex] (Up11) -- (Cp21);
\draw[right hook-latex] (Cp21) -- (J31);
\end{tikzpicture}
\end{center}

\begin{mydefinition}
Since $\pi$ is also a covering map, let $\widetilde{U} \subseteq
\pi^{-1}(\AJ_{\infty}(\mathcal{C}))$ be a neighborhood of $0$ in
$\pi^{-1}(\AJ_{\infty}(\mathcal{C}))$ such that $\pi(\widetilde{U}) = U$ and
$\pi|_{\widetilde{U}} \colon \widetilde{U} \to U$ is an isomorphism. Similarly,
let $\widetilde{U}' \subseteq \pi^{-1}(\mathcal{C}')$ be a neighborhood of $0$
in $\pi^{-1}(\mathcal{C}')$ such that $\pi(\widetilde{U'}) = U'$ and
$\pi|_{\widetilde{U'}} \colon \widetilde{U'} \to U'$ is an isomorphism.

Note that $z \circ \pi$ is a local coordinate on $\widetilde{U}$; we will abuse
notation and denote it by $z$. Similarly, we will abuse notation and write $t$
to be the analogous local coordinate on $\widetilde{U'}$.
\end{mydefinition}

Going to the universal cover yields the following commutative diagram.
\begin{center}
\begin{tikzpicture}
\node(UCp00) at (0, 0) {$\{ 0 \}$};
\node(UCp01) at (0, 2) {$\{ 0 \}$};
\node(UCU10) at (3, 0) {$\widetilde{U}$};
\node(UCUp11) at (3, 2) {$\widetilde{U'}$};
\node(UCC20) at (6, 0) {$\pi^{-1}(\AJ_{\infty}(\mathcal{C}))$};
\node(UCCp21) at (6, 2) {$\pi^{-1} \mathcal{C}'$};
\node(UCJ30) at (9, 0) {$V^{\vee}$};
\node(UCJ31) at (9, 2) {$V^{\vee}$};
\draw[->] (UCp01) -- node[right] {$(1 - \zeta)_{*}$} (UCp00);
\draw[->] (UCUp11) -- node[right] {$(1 - \zeta)_{*}$} (UCU10);
\draw[->] (UCCp21) -- node[right] {$(1 - \zeta)_{*}$} (UCC20);
\draw[->] (UCJ31) -- node[right] {$(1 - \zeta)_{*}$} (UCJ30);
\draw[right hook-latex] (UCp00) -- (UCU10);
\draw[right hook-latex] (UCU10) -- (UCC20);
\draw[right hook-latex] (UCC20) -- (UCJ30);
\draw[right hook-latex] (UCp01) -- (UCUp11);
\draw[right hook-latex] (UCUp11) -- (UCCp21);
\draw[right hook-latex] (UCCp21) -- (UCJ31);
\end{tikzpicture}
\end{center}

\begin{mylemma} \hfill
\label{Lemma:UWiLocalCoordinateExpansion}
\begin{enumerate}[label=\upshape{(\arabic*)},
ref={\autoref{Lemma:UWiLocalCoordinateExpansion}(\arabic*)}]

\item \label{Lemma:UWiLocalCoordinateExpansionU}
The following equality holds in $\widetilde{U}$:
\[
u_{w_{i}} = w_{i}^{-1} z^{w_{i}} (1 + O(z)).
\]

\item \label{Lemma:UWiLocalCoordinateExpansionUp}
The following equality holds in $\widetilde{U'}$:
\[
u_{w_{i}} = w_{i}^{-1} (1 - \zeta^{-b_{i}})^{-1} t^{w_{i}} (1 + O(t)).
\]
\end{enumerate}
\end{mylemma}

\begin{proof} \hfill
\begin{enumerate}[label=\upshape{(\arabic*)},
ref={the proof of \autoref{Lemma:UWiLocalCoordinateExpansion}(\arabic*)}]

\item \label{LemmaProof:UWiLocalCoordinateExpansionU}
Let $\kappa \in V$. Suppose that $P \in \mathcal{C}$ is such that
$\AJ_{\infty}(P) \in U$ and $\gamma$ is a path from $\infty$ to $P$ that lies in
$U$. Since $U$ is simply-connected, the value of the integral $\int_{\gamma}
\kappa$ is independent of the choice of $\gamma$ (as long as $\gamma$ is
contained in $U$), so we will denote this integral by $\int_{\infty}^{P}
\kappa$. By \autoref{Theorem:AbelJacobiWithIntegrals},
\[
\xi \left( \AJ_{\infty}(P) \right) = \left( \kappa \in V \mapsto
\int_{\infty}^{P} \kappa \in \mathbf{C} \right) \pmod{\Lambda} \in V^{\vee} /
\Lambda.
\]
Since $U$ is simply-connected, we may lift this equality to $\widetilde{U}$;
that is, there exists some $\lambda \in \Lambda$ such that for all
$v \in \widetilde{U}$, 
\[
v = \lambda + \left( \kappa \mapsto \int_{\infty}^{\AJ_{\infty}^{-1}(\pi(v))}
\kappa \right).
\]
Taking $v = 0$ shows that $\lambda = 0$. Hence, by definition of
$u_{w_{i}}$,
\begin{equation}
\label{Equation:LocalCoordinateInTermsOfBasis}
u_{w_{i}}( v ) =
\int_{\infty}^{\AJ_{\infty}^{-1}(\pi(v))} \kappa_{i} \text{ for all
} v \in \widetilde{U}.
\end{equation}
Let $P = \AJ_{\infty}^{-1}\left( \pi(v) \right)$, so that
\autoref{Corollary:Coordinatekappai} implies
\begin{equation}
\label{Equation:EvaluateIntegralOnBasisForm}
\int_{\infty}^{P} \kappa_{i} = \int_{0}^{z(P)} z^{w_{i} - 1} (1 + O(z)) \, d z =
w_{i}^{-1} (z(P))^{w_{i}} \left( 1 + O(z(P)) \right).
\end{equation}
Since $z(v)$ was defined to be $z(P)$, we are done by
\eqref{Equation:LocalCoordinateInTermsOfBasis} and
\eqref{Equation:EvaluateIntegralOnBasisForm}.

\item \label{LemmaProof:UWiLocalCoordinateExpansionUp}
\autoref{Lemma:CoordinateFunctionPullbackZeta} implies that for all $v \in
V^{\vee}$, 
\begin{equation}
\label{Equation:TransformCoordinateBy1mz}
u_{w_{i}}(v) = (1 - \zeta^{-b_{i}})^{-1} u_{w_{i}}( (1 - \zeta)_{*} v),
\end{equation}
so since $\widetilde{U} = (1 - \zeta)_{*} \widetilde{U'}$,
\autoref{Lemma:UWiLocalCoordinateExpansionUp} is a consequence of
\eqref{Equation:TransformCoordinateBy1mz}, the definition of $t$, and
\autoref{Lemma:UWiLocalCoordinateExpansionU}.  
\end{enumerate}
\end{proof}

Suppose that $D \in \mathcal{C}' \cap \Theta$. Let $e_{D} \in \pi^{-1}(D)$.

\begin{mydefinition}
Define $\theta$, $\Delta$, $\delta$ as on page 5208 of \cite{Nakay2016} to be the
theta function, the Riemann divisor, and Riemann's constant, respectively.
(Nakayashiki mentions on the same page that $\delta = \Delta - (g - 1) \infty$.)
Then $\delta \in \mathcal{J}$, so let $e_{\delta} \in \pi^{-1}(\delta)$. 
\end{mydefinition}

\begin{mydefinition}
For $F \in \mathcal{J}$, let $T_{F} : \mathcal{J} \to \mathcal{J}$ be the
``translation by $F$'' map. For $e \in V^{\vee}$, let $T_{e} : V^{\vee} \to
V^{\vee}$ be the ``translation by $e$'' map.
\end{mydefinition}

\begin{mytheorem}
\label{Theorem:RiemannVanishingTheorem}
The vanishing locus of $\theta$ is $(\pi \circ T_{e_{\delta}})^{-1} \Theta$.
\end{mytheorem}
\begin{proof}
Riemann's vanishing theorem (see pages 6--7 of \cite{Fay_1973}) states that
the vanishing locus of $\theta$ is $\pi^{-1} T_{\delta}^{-1} \Theta =
(T_{\delta} \circ \pi)^{-1} \Theta$. Since $T_{\delta} \circ \pi = \pi \circ
T_{e_{\delta}}$, we are done.
\end{proof}

\begin{mycorollary}
\label{Corollary:RestateiDvanishing}
$i(D)$ is the order of vanishing of $(\theta \circ T_{e_{D} -
e_{\delta}})|_{\widetilde{U'}}$ at $0$.
\end{mycorollary}
\begin{proof}
By definition, $i(D)$ is the intersection multiplicity of $\Theta$ and
$\mathcal{C}'$ at $D$. Since $\pi \circ T_{e_{\delta}}$ is a local
diffeomorphism at $e_{D} - e_{\delta}$, we know that $i(D)$ is the intersection
multiplicity of $(\pi \circ T_{e_{\delta}})^{-1} \Theta$ and $(\pi \circ
T_{e_{\delta}})^{-1} \mathcal{C}'$ at $e_{D} - e_{\delta}$, so by
\autoref{Theorem:RiemannVanishingTheorem}, 
\begin{equation}
\label{Equation:iDoovRiemannVanishing}
i(D)\text{ is the order of vanishing of }\theta|_{(\pi \circ
T_{e_{\delta}})^{-1} \mathcal{C}'}\text{ at }e_{D} - e_{\delta}. 
\end{equation}
Since $D \in \mathcal{J}[1 - \zeta]$, $T_{D}(U')$ is a neighborhood of $D$ in
$\mathcal{C}'$, so $T_{e_{D} - e_{\delta}}(\widetilde{U'})$ is a neighborhood of
$e_{D} - e_{\delta}$ in $(\pi \circ T_{e_{\delta}})^{-1} \mathcal{C}'$, so
\eqref{Equation:iDoovRiemannVanishing} yields
\[
i(D)\text{ is the order of vanishing of }\theta|_{T_{e_{D} -
e_{\delta}}(\widetilde{U'})}\text{ at }e_{D} - e_{\delta}, 
\]
which is equivalent to
\[
i(D)\text{ is the order of vanishing of }(\theta \circ T_{e_{D} -
e_{\delta}})|_{\widetilde{U'}}\text{ at }0 
\]
since translation by $e_{D} - e_{\delta}$ is an isomorphism on $V^{\vee}$.
\end{proof}

\begin{mydefinition}
As on page 5211 of \cite{Nakay2016}, let $s_{\lambda_{D}} \in
\mathbf{Q}[t_{1}, t_{2}, \cdots]$ be the Schur function associated to the
partition $\lambda_{D}$. Nakayashiki proves that $s_{\lambda_{D}}$ lies in the
subring $\mathbf{Q}[t_{w_{1}}, \cdots, t_{w_{g}}]$ (Proposition 1 on page 5211
of \cite{Nakay2016}). Assign weight $w_{i}$ to the variable $t_{w_{i}}$. Then
$s_{\lambda_{D}}$ is weight-homogeneous and it has weight $|\lambda_{D}|$.
\end{mydefinition}

\begin{myproposition}
\label{Proposition:IntMultAtleastInflectWt}
For all $D \in \mathcal{C}' \cap \Theta$,
\[
i(D) \ge |\lambda_{D}|.
\]
\end{myproposition}
\begin{proof}
Applying Theorem 10 on page 5232 of \cite{Nakay2016} to $e = e_{D} -
e_{\delta}$ (the period matrix $2 \omega_{1}$ defined on page 5231 of
\cite{Nakay2016} is the identity matrix in our application) shows that there is
a nonzero constant $C$ such that for all $u \in V^{\vee}$,
\[
C \theta(u + e_{D} - e_{\delta}) = s_{\lambda_{D}}(t)|_{t_{w_{i}} = u_{w_{i}}} +
\text{higher weight terms},
\]
which we rewrite as
\[
(\theta \circ T_{e_{D} - e_{\delta}}) (u) = C^{-1}
s_{\lambda_{D}}(t)|_{t_{w_{i}} = u_{w_{i}}} + \text{higher weight terms}.
\]
Restricting to $u \in \widetilde{U'}$ and applying
\autoref{Lemma:UWiLocalCoordinateExpansionUp} yields
\[
\theta \circ T_{e_{D} - e_{\delta}} = O(t^{|\lambda_{D}|}) \;
\text{ on } \widetilde{U'},
\]
so we are done by \autoref{Corollary:RestateiDvanishing}.
\end{proof}

\begin{proof}[Proof of \autoref{Theorem:Intersection-Multiplicities}]
Combining \autoref{Lemma:W-Sum} and \autoref{Corollary:I-Sum} yields
\[
\sum_{D \in \mathcal{C}' \cap \Theta} (i(D) - |\lambda_{D}|) \le 0,
\]
so we are done by \autoref{Proposition:IntMultAtleastInflectWt}.
\end{proof}

\section{Acknowledgements}
I would like to thank my advisor, Bjorn Poonen, for suggesting the problem and
for his guidance. I would also like to thank Aaron Pixton for helpful
conversations regarding the computations of the intersection multiplicities. In
addition, I would like to thank the referee for a careful review of the paper
and helpful suggestions.

This research was supported in part by a grant from the Simons Foundation
(\#402472 to Bjorn Poonen).

\end{document}